\documentclass[a4paper]{article}
\usepackage{a4wide}
\usepackage{authblk}
\providecommand{\keywords}[1]{\smallskip\noindent\textbf{\emph{Keywords:}} #1}

\providecommand{\msc}[1]{\smallskip\noindent\textbf{\emph{2020 Mathematics Subject Classification:}} #1}

\usepackage{graphicx}
\usepackage{tikz}
\usetikzlibrary{fit,patterns,decorations.pathmorphing,decorations.pathreplacing,calc}

\usepackage{amsmath,amssymb,amsthm}
\usepackage[inline]{enumitem}

\usepackage{thmtools}

\usepackage{hyperref}
\usepackage[capitalize,nameinlink,sort]{cleveref}

\usepackage[T1]{fontenc}
\usepackage{newpxtext}
\usepackage{euler}

\DeclareMathOperator{\Fac}{\mathcal{L}}
\DeclareMathOperator{\N}{\mathbb{N}}
\DeclareMathOperator{\Z}{\mathbb{Z}}

\DeclareMathOperator{\SL}{SL}

\newcommand{\infw}[1]{\mathbf{#1}}
\newcommand{\fc}[1]{\mathsf{p}_{#1}}
\newcommand{\bc}[2]{\mathsf{b}_{#1}^{(#2)}}

\DeclareMathOperator{\tmprop}{\mathcal{T}\!\!\mathcal{M}\!\mathcal{B}}

\declaretheorem[numberwithin=section]{theorem}
\declaretheorem[sibling=theorem]{lemma,corollary,proposition}
\declaretheorem[sibling=theorem,style=definition]{example,definition,remark,observation}

\declaretheorem{claim}

\declaretheorem[name=Question,refname={Question,Questions},style=definition]{questions}

\declaretheoremstyle[
    headfont=\normalfont\itshape, 
    bodyfont = \normalfont,
    qed=$\blacksquare$, 
    headpunct={:}]{claimproofstyle} 
\declaretheorem[name={Proof of claim}, style=claimproofstyle, unnumbered]{claimproof}

\crefname{equation}{}{}

\DeclareMathOperator{\suff}{suff}
\DeclareMathOperator{\pref}{pref}

\newcommand{\ver}[1]{\textnormal{\texttt{#1}}}

\newcommand{\assoc}{\models}

\title{Characterizations of families of morphisms and words via binomial complexities\footnote{Markus Whiteland dedicates this paper to the memory of his father Alan Whiteland (1940--2021).}}

\author{Michel Rigo}

\author{Manon Stipulanti\thanks{Supported by the FNRS Research grant 1.B.397.20F.}}

\author{Markus A. Whiteland\thanks{Supported by the FNRS Research grant 1.B.466.21F.}}

\affil{Department of Mathematics, University of Li\`ege, Li\`ege, Belgium}

\affil{\texttt{\{m.rigo,m.stipulanti,mwhiteland\}@uliege.be}}

\date{}

\begin{document}
\maketitle            
\begin{abstract}
  Two words are $k$-binomially equivalent if each subword of length
  at most $k$ occurs the same number of times in both words. The
  $k$-binomial complexity of an infinite word is a counting function
  that maps $n$ to the number of $k$-binomial equivalence classes
  represented by its factors of length $n$. Cassaigne et al.~[Int. J. Found. Comput. S., 22(4) (2011)] characterized a family of morphisms, which we call Parikh-collinear, as those morphisms
  that map all words to words with bounded $1$-binomial
  complexity. Firstly, we extend this characterization: they map words
  with bounded $k$-binomial complexity to words with bounded
  $(k+1)$-binomial complexity.  As a consequence, fixed points of
  Parikh-collinear morphisms are shown to have bounded $k$-binomial
  complexity for all $k$. Secondly, we give a new characterization of Sturmian
  words with respect to their $k$-binomial complexity. Then we
  characterize recurrent words having, for some $k$, the same $j$-binomial
  complexity as the Thue--Morse word for all $j\le k$. Finally, inspired by
  questions raised by Lejeune, we study the relationships between the
  $k$- and $(k+1)$-binomial complexities of infinite words; as well as
  the link with the usual factor complexity.
  

\msc{Primary: 68R15. Secondary: 05A05.}
  \end{abstract}

%
%
%

\section{Introduction}\label{sec: intro}

The combinatorial structure of an infinite word $\infw{x}\in A^\mathbb{N}$ over a finite alphabet~$A$ may reveal important aspects of $\infw{x}$ itself. 
This structure is often studied through its language $\mathcal{L}(\infw{x})$, i.e., the set of its factors, and in particular to inspect the set $\Fac_n(\infw{x}):=\mathcal{L}(\infw{x})\cap A^n$ of factors of length~$n$.
Even plain counting the cardinality of this set turns out to be a useful concept: with $\fc{\infw{x}}$ the \emph{factor complexity}
function defined as $\fc{\infw{x}}\colon\N \to \N$, $n\mapsto \#\Fac_n(\infw{x})$, the celebrated Morse--Hedlund theorem asserts that
an infinite word $\infw{x}$ is aperiodic if and only if $\fc{\infw{x}}(n) \geq n+1$ for all $n\geq 1$. 
For instance,  the Thue--Morse word $\infw{t}=01101001\cdots$ (also known as the Prouhet--Thue--Morse word),  the fixed point of the morphism $\varphi \colon 0\mapsto 01, 1\mapsto 10$,  is aperiodic because its factor complexity is given by
\begin{equation}\label{eq: fact-compl-TM}
\fc{\infw{t}}(2^m + r) =
\begin{cases}
3\cdot 2^m + 4(r-1), & \text{if } 1 \leq r \leq 2^{m-1};\\
4\cdot 2^m + 2(r-1), & \text{if } 2^{m-1}<r\leq 2^m,
\end{cases}
\end{equation}
for all $m\geq 0$ \cite[\S 4]{CANT2010}.
The factor complexity has proved its importance in a number of areas of mathematics. 
For example,  in number theory,  Adamczewski and Bugeaud \cite{AdamczewskiBugeaud2007} proved that the base-$b$ expansion of a real algebraic irrational number has a factor complexity function satisfying
\[
\liminf_{n \to \infty} \frac{\fc{}(n)}{n} = \infty.
\]
As a consequence,  the number having $\infw{t}$ as base-$2$ expansion is transcendental.

One can conversely define families of words using the factor complexity function. For example, a word $\infw{x}$ is called \emph{Sturmian} if $\fc{\infw{x}}(n) = n+1$ for all $n$. Such words, studied also in this note, turn out to have many interesting properties. For general references about combinatorics on (Sturmian) words, see, for
instance, \cite{AS,CANT2010,Lothaire1997}.

Many variations of the factor complexity have been introduced.
Some of these counting functions only take into account factors
with specific properties such as palindromes or privileged words
\cite{DroubayPirillo1999,Peltomaki2013privileged}. Other
functions count subwords extracted along subsequences of
prescribed forms like maximal pattern or arithmetical
complexities \cite{KamaeZamboni2002,Avgustinovich2003}. Closely
related to the subject discussed in this paper, abelian,
$k$-abelian or cyclic complexities are functions of the form
$n\mapsto\#(\Fac_n(\infw{x})/{\sim})$ for the quotient by a
relevant equivalence relation~$\sim$ \cite{Rigo2015,KarhumakiSaarelaZamboni2013,Cyclic}.
For more on abelian combinatorics on words (and related notions),
we refer the reader to the recent excellent survey
\cite{FiciPuzynina}. For each of the above complexity functions
usual questions naturally arise: 

\begin{itemize}
\item What is the complexity of well-known families of words such as Sturmian, Arnoux-Rauzy, automatic, (pure) morphic, or Toeplitz words? 

\item This leads to the more interesting problem of classifying or characterizing infinite words with respect to their complexity.
As an example, Coven and Hedlund in \cite{CovenH1973sequences} show that an infinite word is purely periodic if and only if its abelian complexity function attains the value $1$. Further, a binary aperiodic word has its abelian complexity function equal to the constant function $2$ if and only if it is Sturmian. Words with linear factor complexity are characterized in \cite{CassaigneFPZ2019characterization}.

\item What are the possible growth rates of the complexity function? 

\item Which non-periodic words may achieved the lowest complexity? 
\end{itemize}
These questions have intrinsic theoretical interests but also provide particular insight about the combinatorial structure of the studied words. Depending on the properties of interest, one focuses on the appropriate complexity function. Infinite words with specific combinatorial properties are, for instance, sought to construct particular symbolic dynamical systems or tilings of the line. For instance, the Thue--Morse minimal subshift is completely characterized by its factor complexity together with its abelian complexity \cite{RichommeSZ2011Abelian}.
The Thue--Morse word, which is central in our paper,  plays an important role in many areas of mathematics, e.g., see \cite{AlloucheS1999ubiquitous,Allouche2015}. 

In this paper, the complexity function of interest is built from binomial coefficients of words.
\begin{definition}
Let $u,w\in A^*$.
The \emph{binomial coefficient} of $u$ and $w$ is the number of times $w$ occurs as a subword of $u$, i.e.
writing $u = u_1\cdots u_n$ with $u_i \in A$,
$$\binom{u}{w}=\#\left\{ i_1<i_2<\cdots < i_{|w|} : u_{i_1}u_{i_2}\cdots u_{i_{|w|}}=w\right\}.$$
\end{definition}
These binomial coefficients have proven to be useful in a variety of domains: generalizations of Pascal-like triangles \cite{LRSti2016}, algebra and topology \cite{PinSilva,CrochemoreBerstel}, formal languages or relationship with the extensively studied Parikh matrices and Simon's congruence \cite{Atanasiu2002,Fosse2004,SimonTest}. For more on these binomial coefficients, see, for instance,~\cite[\S 6]{Lothaire1997}.

We mention a well-known and actively researched problem related to binomial coefficients. A word $u$ is {\em $k$-reconstructible}
whenever the knowledge of the binomial coefficients
$\binom{u}{v}$, for all subwords $v$ of length~$k$, uniquely
determines $u$. Inspired by the problem of reconstructing graphs
from vertex-deleted subgraphs, the famous reconstruction problem
is to determine the function $f(n)=k$ where $k$ is the least
integer for which all words of length~$n$ (over a given alphabet)
are $k$-reconstructible. For articles on this problem, we mention \cite{FleischmannLejeune2021,%
Kalasnik,ManvelMeyerowitz}
and references therein.

Let us now introduce our main object of study. Let $k\ge 1$ be an integer. The $k$-binomial complexity function introduced in \cite{RigoSalimov2015} is the central theme  of Lejeune's thesis \cite{Lejeune2021thesis}. It is built on the $k$-binomial equivalence where factors are distinguished with respect to the number of occurring subwords.
\begin{definition}
Two words $u, v\in A^*$ are \emph{$k$-binomially equivalent}, and we write $u \sim_k v$, if
\[
\binom{u}{x} = \binom{v}{x}, \quad \forall\, x\in A^{\le k}.
\]
\end{definition}
%
In \cite[Lem.~1]{ManvelMeyerowitz}, it is observed that one may replace
the condition $\forall x\in A^{\leq k}$ with $\forall x\in A^k$
as soon as $|u|$, $|v| \geq k$.
Observe that the word $u$ is obtained as a permutation of the letters in $v$ if and only if $u \sim_1 v$. The latter relation is the \emph{abelian equivalence} already introduced by Erd\H{o}s \cite{Erdos1957}. This leads to the following definition, introduced in \cite{RigoSalimov2015}.
\begin{definition}
Let $k\ge 1$ be an integer. The \emph{$k$-binomial complexity function} of an infinite word $\infw{x}$ is defined as $\bc{\infw{x}}{k} \colon \N \to \N$, $n\mapsto \#(\Fac_n(\infw{x})/{\sim_k})$.
\end{definition}

It is clear that we have a series of refinements of the abelian equivalence: for all $k \ge 1$,
$u \sim_{k+1} v$ implies $u\sim_k v$. Thus, for all~$n$, we have the
inequalities
\begin{equation}
  \label{eq:ineq}
\bc{\infw{x}}{1}(n) \le \bc{\infw{x}}{2}(n) \le \cdots \le \bc{\infw{x}}{k}(n) \le \bc{\infw{x}}{k+1}(n) \le \cdots \le \fc{\infw{x}}(n).  
\end{equation}

The study of the $k$-binomial complexity function has so far been
studied for restricted families of words. For example, for
$k\ge 2$, the $k$-binomial complexity of Sturmian words coincides
with their factor complexity \cite{RigoSalimov2015} (recalled here as \cref{thm:Sturmian2-binomial}) and the same property holds for the Tribonacci word \cite{LejeuneRigoRosenfeld2020}.
For any $k\geq 2$, fixed points of Parikh-constant morphisms (see the next part for a definition) are known to have bounded
$k$-binomial complexity \cite{RigoSalimov2015}. Recently, the
$k$-binomial complexities of the Thue--Morse word
\cite{LejeuneLeroyRigo2020} (given in \cref{eq:k-image-j-complexities}) and the $2$-binomial complexities of generalized Thue--Morse words was also computed \cite{2021binomial}. That is the extent to which the notion has been studied.

We remark that a better understanding of the $k$-binomial complexity may give information about the language $\mathcal{L}(\infw{x})$ of an infinite word~$\infw{x}$ for which the reconstruction problem could be solved. The aim is to restrict the reconstruction problem to the language of an infinite word having a $k$-binomial complexity of the same order as its factor complexity. Indeed, if
$\bc{\infw{x}}{k}=\fc{\infw{x}}$ for some $k$, then for any two distinct factors $y,z$ of $\infw{x}$, there exists a subword $v$ of length $k$ such that $\binom{y}{v}\neq\binom{z}{v}$. 

Finally a parallel can be drawn between the $k$-abelian complexity introduced by Karhum\"aki et al.~\cite{KarhumakiSaarelaZamboni2013} and the
$k$-binomial complexity. In both cases, we have a series of refinements \cref{eq:ineq} of the abelian equivalence. The fundamental difference is the following one. Two finite words $u,v$ are {\em $k$-abelian equivalent}
if, for each word $w$ of length at most~$k$, we count the same number of occurrences of the factor $w$ in both words $u$ and~$v$.
We thus make the important distinction between a \emph{factor} and
a \emph{subword} of a word. Many properties of the $k$-abelian complexity have been
recently and extensively studied such as growth and fluctuations, $k$-abelian palindromes,
variation of Morse--Hedlund theorem, etc.~\cite{CassaigneKarhumaki2017,CassaigneKarhumaki2018,KarhumakiSaarela2017}.
This is to be contrasted with the limited knowledge we have on
the $k$-binomial complexity function. Indeed, part of our motivation for this work stems from this rather limited state of the art as described above.

\subsection{Our Results}
We present three kinds of results: a new characterization of
Parikh-collinear morphisms and links with bounded binomial
complexities; a characterization of recurrent words with the same
$j$-binomial complexities as the Thue--Morse word for
$j=1,\ldots,k$; study of the relationships existing between $\bc{\infw{w}}{k}$
and $\bc{\infw{w}}{k+1}$.
This paper improves by far the preliminary version \cite{RSWDLT}: not only do we prove announced results,  but we also significantly extend them.
\smallskip

$\bullet$ Morphisms mapping all infinite words to words with bounded
abelian complexity have been characterized in \cite{CassaigneRichomeSaariZamboni2011}.
Such a morphism $f\colon A^* \to B^*$ is said to be \emph{Parikh-collinear}: for all
letters $a,b \in A$, there is $r_{a,b}\in\mathbb{Q}$ such that $\Psi(f(b)) = r_{a,b} \Psi(f(a))$,
where $\Psi(u)$ denotes the \emph{Parikh vector} of a word~$u$ (see \cref{sec:prelim} for definitions).
In \cref{sec: rank1}, we obtain several new characterizations of Parikh-collinear morphisms.  Connecting this with the series of inequalities \cref{eq:ineq}, we show with \cref{thm:Parikh-collinear_bounded-k_to_bounded_k+1} that a morphism is Parikh-collinear if and only if it maps all words with bounded $k$-binomial complexity to words with bounded $(k+1)$-binomial complexity. 

It is known that any fixed point of a prolongable {\em Parikh-constant morphism}~$f:A^*\to A^*$, i.e., $\Psi(f(a))=\Psi(f(b))$ for all letters $a,b\in A$, has a bounded $k$-binomial complexity \cite{RigoSalimov2015}. Any Parikh-constant morphism is obviously Parikh-collinear. As a direct consequence of our characterization of Parikh-collinear morphisms, \cref{cor:k-binomial_bounded} extends the previous result: bounded $k$-binomial complexity holds for any fixed point of a prolongable Parikh-collinear morphism. 
\smallskip

$\bullet$
We now turn to words sharing their binomial complexities with the Thue--Morse word $\infw{t}$.
From the above discussion (the Thue--Morse morphism $\varphi$ is Parikh-constant), for all $j\ge 1$, the $j$-binomial complexity of $\infw{t}$ is bounded by a constant depending on~$j$. But more is known, the exact value of $\bc{\infw{t}}{j}(n)$ computed in \cite{LejeuneLeroyRigo2020} is given by 
\begin{equation}\label{eq:k-image-j-complexities}
\bc{\infw{t}}{j}(n)=
\begin{cases}
\fc{\infw{t}}(n), & \text{if } n < 2^j;\\
3\cdot 2^j-3, & \text{if } n\equiv 0 \pmod{2^j} \text{ and } n \geq 2^j; \\
3\cdot 2^j-4, & \text{otherwise},
\end{cases}
\end{equation}
where the factor complexity $\fc{\infw{t}}$ of $\infw{t}$
is given by \cref{eq: fact-compl-TM}.
Considering $j=1$ in \cref{eq:k-image-j-complexities}, words having the same abelian complexity as the Thue--Morse words have been characterized in \cite{RichommeSZ2011Abelian} as follows. The abelian complexity of an aperiodic word~$\infw{x}\in\{0,1\}^\mathbb{N}$ is, for $n>0$, $\bc{\infw{x}}{1}(n)=3$ if $n$ is even, and $\bc{\infw{x}}{1}(n)=2$ if $n$ is odd, if and only if there exists a word~$\infw{y}$ such that $\infw{x}=u\varphi(\infw{y})$ with $u\in\{\varepsilon,0,1\}$.
\cref{sec:propsTMI,sec:propsTMII} are about binomial properties of iterates of $\varphi$. We generalize the latter result and obtain a characterization of words having the same $j$-binomial complexity as the Thue--Morse word~$\infw{t}$ for all $j\le k$.
Except for a remark in \cite{FiciPuzynina} (see \cref{thm:fici}), such a result together with \cref{thm:charact_sturmian} are the first where binomial complexity leads to the characterization of combinatorial families of words. In this paper, with \cref{thm:charact_sturmian}, we observe that a word $\infw{x}$ is Sturmian if and only if $\bc{\infw{x}}{2}(n)=n+1$, for all $n$. We make the statements about words sharing the same $j$-binomial complexities as $\infw{t}$ more precise.

Let $k$ be an integer and let $\infw{y}$ be an aperiodic binary
word. With \cref{thm:k-image-j-complexities} we show that for $\infw{x} = u\varphi^k(\infw{y})$ we have, for all $j \leq k$, $\bc{\infw{x}}{j}=\bc{\infw{t}}{j}$ which is given by \eqref{eq:k-image-j-complexities}, where $u$ is a (possibly empty) proper suffix of $\varphi^k(0)$ or $\varphi^k(1)$. Conversely, with \cref{thm:k-complexities-k-image}, if $\bc{\infw{x}}{j}=\bc{\infw{t}}{j}$ for all $j\leq k$ for a \emph{recurrent} word~$\infw{x}$, i.e., each factor of $\infw{x}$ appears infinitely often,
then $\infw{x}=u\varphi^k(\infw{y})$ where $u$ is a proper suffix of $\varphi^k(0)$ or $\varphi^k(1)$ and $\infw{y}$ is some aperiodic binary word.

\smallskip

$\bullet$ In general, not much is known about the general behavior or fluctuations that can be expected for the $k$-binomial complexity of an infinite word. In particular, computing the $k$-binomial complexity of a particular infinite word remains quite challenging. It would also be desirable to compare in some ways the $k$- and $(k+1)$-binomial complexities of a word.
\begin{definition}
For two functions $\mathsf{f}, \mathsf{g} \colon \N \to \N$, we write $\mathsf{f} \prec \mathsf{g}$ when the relation $\mathsf{f}(n) < \mathsf{g}(n)$ holds for infinitely many $n \in \N$.\end{definition}

We define $\prec$ this way because for some words,  the $2$-binomial complexity attains the factor complexity infinitely often while it is less than the factor complexity infinitely often.
See end of \cref{sec: 7.1} for a discussion.

As an example, a consequence of \cref{pro:k-image_k+1-prefix-suffix} is that $\bc{\infw{x}}{k}\prec \bc{\infw{x}}{k+1}$ for $\infw{x}=\varphi^k(\infw{y})$ with $\infw{y}$ aperiodic.
Our reflection is here driven by the following questions inspired by Lejeune's questions \cite[pp.~115--117]{Lejeune2021thesis} that are natural to consider in view of \eqref{eq:ineq}.
\begin{questions}\label{q: strict ineq}
Does there exist an infinite word $\infw{w}$ such that, for all $k\geq 1$, $\bc{\infw{w}}{k}$ is unbounded and $\bc{\infw{w}}{k}
\prec \bc{\infw{w}}{k+1}$?
If the answer is positive, can we find a (pure) morphic such word $\infw{w}$?
\end{questions}
From \eqref{eq:ineq}, notice that $\bc{\infw{w}}{k}$ is unbounded, for all $k\geq 1$, if and only if the abelian complexity $\bc{\infw{w}}{1}$ is unbounded. Even though the Thue--Morse word $\infw{t}$ is such that, for all $k\ge 1$, $\bc{\infw{t}}{k}\prec \bc{\infw{t}}{k+1}$,  $\bc{\infw{t}}{k}$ remains bounded \eqref{eq:k-image-j-complexities}. So $\infw{t}$ is not a satisfying answer to \cref{q: strict ineq}. However, in \cref{sec:question_of_Lejeune}, we provide several positive
answers to  this question.

\begin{restatable}{questions}{qstab}
\label{q: stab}
For each $\ell \ge 1$, does there exist a word $\infw{w}$ (depending on $\ell$) such that $
\bc{\infw{w}}{1}
\prec \bc{\infw{w}}{2}
\prec \cdots 
\prec \bc{\infw{w}}{\ell-1}
\prec \bc{\infw{w}}{\ell}
= \fc{\infw{w}}$? 
If the answer is positive, is there a (pure) morphic such word $\infw{w}$?
\end{restatable}


Putting together results from \cref{sec: answer question stab,q: stab-bis} we fully answer \cref{q: stab}: \cref{thm:k-image-j-complexities} and \cref{pro:k-image_k+1-prefix-suffix} provide a word $\infw{x}=\varphi^k(\infw{y})$ for which
$\bc{\infw{x}}{1}
\prec \bc{\infw{x}}{2}
\prec \cdots 
\prec \bc{\infw{x}}{k-1}
\prec \bc{\infw{x}}{k}
\prec \bc{\infw{x}}{k+1}$, while assuming that
$\infw{y}$ above is Sturmian, we show that $\bc{\infw{x}}{k+2}=\fc{\infw{x}}$.
We remark that iterates of $\varphi$ applied to Sturmian words
have been studied (among other words) in \cite{Frid1999}. We
observe that our construction leads to words with bounded abelian
complexity. \cref{q: stab} is then strengthened in
\cref{q: stab-bis} where we ask for words with unbounded abelian
complexity. We give a pure morphic answer when $\ell=3$.

\section{Preliminaries}\label{sec:prelim}

Let us now give precise definitions and notation. For any integer $k$, we let $A^k$ (resp., $A^{\le k}$; resp., $A^{< k}$) denote the set of words of length exactly (resp., at most; resp., less than) $k$ over $A$. We let $A^*$ (resp., $A^+$) denote the semigroup of finite words (resp., non-empty finite words) over $A$ equipped with concatenation. We let $\varepsilon$ denote the empty word. The length of the word $w$ is denoted by $|w|$ and the number of occurrences of a letter $a$ in $w$ is denoted by $|w|_a$. 
For binary words $u$, $v$ (always over $\{0,1\}$ in this note, unless otherwise stated), we refer to $|u]_1$ as the \emph{weight} of $u$ and we say that $u$ is {\em lighter} (resp., {\em heavier}) than $v$ whenever $|u|_1<|v|_1$ (resp.,
$|u|_1>|v|_1$). For instance, if $\bc{\infw{y}}{1}(n)=2$, then there are only two kinds of factors in $\infw{y}$: the light ones and the heavy ones. A language $L$ is said to be {\em balanced} if, for all words $u,v\in L$ of the same length and all letters $a$, we have $\bigl| |u|_a-|v|_a\bigr|\le 1$. In particular, an infinite word $\infw{z}$ is {\em balanced} if $\Fac(\infw{z})$ is balanced.

We let $\overline{\,\cdot\,}$ denote the (binary) complementation morphism defined by $\overline{a} = 1-a$, for $a\in\{0,1\}$.
Writing $A=\{a_1,\ldots,a_k\}$ and fixing the order $a_1 < a_2 < \cdots < a_k$ on the letters, the \emph{Parikh vector} of a word $w\in A^*$ is defined as the column vector
\[
\Psi(u) = 
(|w|_{a_1}, 
|w|_{a_2},
\ldots,
|w|_{a_k}
)
^\intercal.
\]

Using a classical ``length-$n$ sliding window'' argument or extending factors of length~$n$ to factors of length~$n+1$, one has the following.

\begin{lemma}[Folklore]\label{rk:abelian_complexity}
  For any binary word $\infw{y}$ over $\{0,1\}$, for all $n\ge 0$, we have $$\bc{\infw{y}}{1}(n)=1+\max_{u,v\in \Fac_n(\infw{y})}\bigl| |u|_1-|v|_1\bigr| \quad\text{ and }\quad
  \bigl| \bc{\infw{y}}{1}(n+1)-\bc{\infw{y}}{1}(n)\bigr|\le 1.$$
\end{lemma}

\subsection{Binomial Equivalence}
We first collect some useful results on $k$-binomial equivalence. Note that $\sim_k$ is a congruence, i.e., for $u,v,x\,y\in A^*$, $u\sim_k v$ and $x\sim_k y$ implies $ux\sim_k vy$. In particular,
$A^*/{\sim_k}$ is a monoid. In fact, it is a cancellative monoid
(see \cite[Lemma~10]{LejeuneLeroyRigo2020}; cancellativity also follows from $A^*/{\sim_k}$ being isomorphic to a subsemigroup
of the special linear group $\SL((k+1)\cdot\#A^k,\Z)$ \cite{RigoSalimov2015}):

\begin{lemma}[Cancellation property]\label{lem: cancellation property} Let $u, v, w$ be words over $A$. We have $$v\sim_k w\Leftrightarrow uv\sim_kuw \text{ and } v\sim_k w\Leftrightarrow vu\sim_k wu.$$
\end{lemma}

We will also need the following result characterizing $k$-binomial commutation among words of equal length.

\begin{theorem}[{\cite[Thm.~3.5]{Whiteland2021}}]\label{thm:w35}
Let $k\geq 2$ and $x,y\in A^*$ such that $|x|=|y|$. Then $xy\sim_k yx$ if and only if $x\sim_{k-1}y$.
\end{theorem}

A proof of the next result can be conveniently found in \cite[Lem.~30]{LejeuneLeroyRigo2020}.

\begin{theorem}[Ochsenschl\"ager \cite{Ochsenschlager}]\label{thm:Ochsenschlager}
Let $\varphi\colon 0\mapsto 01, 1\mapsto 10$ be the Thue--Morse morphism. For all $k\ge 1$, we have
$\varphi^k(0) \sim_k \varphi^k(1)$ and $\varphi^k(0) \not\sim_{k+1} \varphi^k(1)$.
\end{theorem}

The following result from \cite[Lem.~31]{LejeuneLeroyRigo2020} will be of use. It can alternatively be proved using \cref{thm:w35} combined with Ochsenschl\"ager's result.
\begin{lemma}[Transfer lemma]\label{lem:transfer}
  Let $k\ge 1$. Let $u,v,v'$ be three non-empty words such that $|v|=|v'|$. We have $\varphi^{k-1}(u)\varphi^k(v)\sim_k\varphi^k(v')\varphi^{k-1}(u)$.
\end{lemma}

It is an exercise to see that, for an arbitrary morphism $f\colon A^* \to B^*$, we have, for all $u \in A^*$, $e \in B^*$, 
\begin{equation}\label{eq:morphic_image_coeff}
\binom{f(u)}{e} = \sum_{\substack{a_1,\ldots,a_{\ell} \in A\\ \ell \leq |e|} }\binom{u}{a_1\cdots a_{\ell}}
						\sum_{\substack{e=e_1\cdots e_{\ell}\\e_i\in B^+}}\prod_{i=1}^{\ell}\binom{f(a_i)}{e_i}.
\end{equation}

The next result will turn out to be useful in several places of the paper.

\begin{lemma}\label{lem:diff-powers}
Let $x,y \in A^*$ be two $k$-binomially equivalent words.
For any integer $n\ge 0$ and any word $e\in A^*$ of length $k+1$, we have
\[
\binom{x^n}{e} - \binom{y^n}{e} = n \left[ \binom{x}{e} -\binom{y}{e} \right].
\]
In particular, for all $n\ge 1$, $x \sim_{k+1} y$ if and only if $x^n \sim_{k+1} y^n$.
\end{lemma}
\begin{proof}
For any words $u,v,w\in A^*$, we have
\[
\binom{uv}{w} = 
	\binom{u}{w}+\binom{v}{w} + \sum_{\substack{w = w_1w_2\\ w_i \neq \varepsilon}}
		\binom{u}{w_1}\binom{v}{w_2}. 
\]
To show the statement, we proceed by induction and we make use of the previous formula.
The statement is trivially true for $n\in\{0,1\}$.
By the previous formula (with $(u,v,w)=(x^n,x,e)$ and $(u,v,w)=(y^n,y,e)$ respectively) and the induction hypothesis, we obtain
\begin{align*}
\binom{x^{n+1}}{e} - \binom{y^{n+1}}{e}
&= (n+1) \left[ \binom{x}{e} -\binom{y}{e} \right] 
+ \sum_{\substack{e = e_1e_2\\ e_i \neq \varepsilon}}
		\left[\binom{x^n}{e_1}\binom{x}{e_2} -
		\binom{y^n}{e_1}\binom{y}{e_2} \right].
\end{align*}
Since $x \sim_k y$ and $|e|=k+1$, the sum in the right-hand term is zero and we obtain the desired result.
\end{proof}

We recall the following lemma that appears in~\cite{Whiteland2021}; it is a straightforward generalization of an observation in~\cite{Salomaa2003}.
We give a proof for the sake of completeness.

\begin{lemma}\label{lem:sum-constantPvect}
Let $\mathcal{C} \in A^*/{\sim_1}$ be an abelian equivalence class of non-empty words with Parikh vector $(m_a)_{a\in A}$.
Then, for any word $u \in A^*$, we have $\sum_{w \in \mathcal{C}}\binom{u}{w} = \prod_{a \in A} \binom{|u|_a}{m_a}$.
\end{lemma}
\begin{proof}
The sum on the left counts the number of ways one can choose a subword $w$ of $u$ so that $\Psi(w) = (m_a)_{a\in A}$.
On the other hand, for a vector $(m_a)_{a \in A}$, any choice of $m_a$ many distinct $a$'s in $u$ for each $a\in A$ gives
rise to a subword of $u$ having Parikh vector $(m_a)_{a\in A}$. The number of distinct such choices is the product on the right.
\end{proof}

\subsection{Binomial equivalence in Sturmian words}
The following result links the factor complexity and the $2$-binomial complexity of Sturmian words.

\begin{theorem}[{\cite[Thm.~7]{RigoSalimov2015}}]\label{thm:Sturmian2-binomial}
For any Sturmian word $\infw{s}$, we have $\bc{\infw{s}}{2} = \fc{\infw{s}}$.
\end{theorem}
In particular, the theorem implies that for two distinct equal-length factors $u$, $v$ of a Sturmian word, we have either $u \not\sim_1 v$, or $\binom{u}{01} \neq \binom{v}{01}$. It further
implies that $\bc{\infw{s}}{k}(n) = n+1$ for all $k\geq 2$.
In the survey paper
\cite{FiciPuzynina} on abelian combinatorics on words, Fici and Puzynina derive a characterization of Sturmian words from \cref{thm:Sturmian2-binomial}:
\begin{theorem}[{\cite[Rem.~80]{FiciPuzynina}}]\label{thm:fici}
\label{thm:FiciPuzynina}
Let $\infw{x}$ be an infinite word. The following are equivalent:
\begin{enumerate}
\item $\infw{x}$ is Sturmian;
\item for all $n\geq 1$ and some $k\geq 2$,
$\bc{\infw{x}}{1}(n) = 2$ and $\bc{\infw{x}}{k}(n) = n+1$;
\item for all $n\geq 1$ and $k\geq 2$,
$\bc{\infw{x}}{1}(n) = 2$ and $\bc{\infw{x}}{k}(n) = n+1$.
\end{enumerate}
\end{theorem}
In fact, the second property in the above can be weakened to
``$\bc{\infw{x}}{1}(n) = 2$ for all $n\geq 1$ and
$\sup_{k,n\in\N} \bc{\infw{x}}{k}(n) = \infty$'' using the same arguments
\footnote{Indeed, the former property implies $\infw{x}$ is balanced and binary,
and the latter implies that $\infw{x}$ is aperiodic.}.
In the following we show that the assumption of balancedness can
be removed from the second point; Sturmian words are characterized by their
$k$-binomial complexity for any fixed integer $k\geq 2$. We first recall
the following crucial observation, which can be found
in part of
\cite[Lem.~4.02'\footnote{The statement has a fourth condition, which does not affect the conclusion appearing here.}]{CovenH1973sequences}

\begin{lemma}%
\label{lem:coven}
  Let $\infw{z}$ be an infinite binary word. Let $N\ge 2$ be such that
  \begin{enumerate}
  \item $\Fac(\infw{z})\cap A^{<N}$ is balanced;  
  \item $\Fac(\infw{z})\cap A^{N}$ is unbalanced;
  \item $\fc{\infw{z}}(N)=N+1$;
  \end{enumerate}
  Then $\infw{z}$ is ultimately periodic.
\end{lemma}

\begin{theorem}\label{thm:charact_sturmian}
Let $\infw{z}$ be an infinite word such that for some $k \geq 2$,
$\bc{\infw{z}}{k}=n+1$ for all $n$. Then $\infw{z}$ is Sturmian. 
\end{theorem}
\begin{proof}
Assume that $\bc{\infw{z}}{k}(n) = n+1$ for all $n\ge 0$. In particular, $\infw{z}$ is binary and also aperiodic because $\fc{\infw{z}}(n) \ge \bc{\infw{z}}{k}(n)=n+1$. This also implies $\bc{\infw{z}}{1}(n)\ge 2$ for all $n\ge 1$. To get a contradiction, assume that $\infw{z}$ is unbalanced. Hence there exists a minimal integer $N\ge 2$ such that $\bc{\infw{z}}{1}(N)=3$. There is a pair  $(u,v)$ of factors of $\infw{z}$ of length $N$ such that $|u|_1-|v|_1=2$. The minimality of $N$ implies that this pair is unique and of the form $(1w1,0w0)$.
For details, see \cite[Lem.~3.06]{CovenH1973sequences}.

Let $|w|_1=r$. Let $x\in\Fac_{N-2}(\infw{z})$. If $|x|_1=r-1$, then $0x$ and $x0$ do not belong to $\Fac_{N-1}(\infw{z})$. Indeed, $1w,w1$ belong to the latter set and $|1w|_1=|w1|_1=r+1$ but by minimality of $N$, the set $\Fac_{N-1}(\infw{z})$ is balanced. In that case, $x$ is preceded and followed by $1$. Similarly, if $y\in\Fac_{N-2}(\infw{z})$ and $|y|_1=r+1$, then $y$ is preceded and followed by $0$. This means that $\Fac_N(\infw{z})\setminus \{0w0,1w1\}$ is a subset of 
$$ \{1x1: |x|_1=r-1, |x|=N-2\}\cup \{0y0: |y|_1=r+1, |y|=N-2\} \cup \bigcup_{a\in\{0,1\}}\{az\overline{a}: |z|_1=r, |z|=N-2\}$$ where all words have weight $r+1$.
Now, observe that $\Fac(\infw{z})\cap A^{<N}$ is a balanced set (by minimality of $N$) and that is also the case of 
$\Fac_N(\infw{z})\setminus \{1w1\}$. The union of these two sets is factorial. By \cite[Thm.~3.1]{RichommeSeebold2011}, there exists a Sturmian word $\infw{s}$ such that
\[
(\Fac(\infw{z})\cap A^{<N}) \cup (\Fac_N(\infw{z})\setminus \{1w1\}) \subset \Fac(\infw{s}).
\]
As a consequence of \cref{thm:Sturmian2-binomial}, any two distinct words in the left-hand side set are not $k$-binomially equivalent. Also, $1w1$ is not abelian (and thus not $k$-binomially) equivalent to any word in $\Fac_N(\infw{z})\setminus \{1w1\}$. In particular, since $\bc{\infw{z}}{k}(n)=n+1$ for $n\le N$, 
\[
\Fac(\infw{z})\cap A^{<N} = \Fac(\infw{s})\cap A^{<N}
\]
and $\#(\Fac_N(\infw{z})\setminus \{1w1\})=N$. Therefore,
$\#(\Fac_N(\infw{z}))=N+1$.
The word $\infw{z}$ now fulfills all the conditions of \cref{lem:coven} implying the contradiction that $\infw{z}$ is ultimately periodic.
\end{proof}

\section{Parikh-Collinear Morphisms via Binomial Complexities}\label{sec: rank1}
In this section, we obtain a new characterization of Parikh-collinear morphisms and show that, given an infinite fixed point of a
prolongable Parikh-collinear morphism, its $k$-binomial complexity is bounded for each $k$. 

\begin{definition}[Parikh-collinear morphisms]
\label{def:Parikh-collinear}
A morphism $f\colon A^* \to B^*$ is said to be \emph{Parikh-collinear} if, for all letters $a,b \in A$, there is $r_{a,b}\in\mathbb{Q}$ such that $\Psi(f(b)) = r_{a,b} \Psi(f(a))$.
\end{definition}

\begin{remark}
Given a morphism $f\colon A^* \to B^*$, its {\em adjacency matrix} $M_f$ is the matrix of size $|B|\times |A|$ defined by $(M_f)_{b,a}=|f(a)|_b$ for all $a\in A$, $b\in B$. Observe that $f$ is a Parikh-collinear morphism if and only if $M_f$ has rank $1$ (unless it is totally erasing). We observe that for any word $u \in A^*$, we have that $\Psi(f(u)) = M_f\Psi(u)$.
\end{remark}

\begin{example}\label{exa:pc}
The morphism $f$ defined by $0 \mapsto 000111$; $1 \mapsto 0110$ is Parikh-collinear since $\Psi(f(1))=\frac{2}{3} \Psi(f(0))$. 
\end{example}


\begin{theorem}[{\cite[Thm.~11]{CassaigneRichomeSaariZamboni2011}}]\label{thm: abelian complexity iff Parikh-collinear}
A morphism $f \colon A^* \to B^*$ maps all infinite words to words with bounded abelian complexity if and only if it is Parikh-collinear.
\end{theorem}

We extend the above theorem to the following one.
We say that a morphism $f\colon A^*\to B^*$ {\em satisfies $P_k$} if $f$ maps all words with bounded $k$-binomial complexity to words with bounded $(k+1)$-binomial complexity.
Note that,  for $k=0$,  {\em $0$-binomial complexity} has to be understood as the ``equal length'' equivalence relation. So the $0$-binomial complexity of an infinite word is the constant function~$1$ and \cref{thm: abelian complexity iff Parikh-collinear} can be restated as \emph{$f$ is Parikh-collinear if and only if $f$ satisfies~$P_0$}.

\begin{theorem}\label{thm:Parikh-collinear_bounded-k_to_bounded_k+1}
Let $f\colon A^*\to B^*$ be a morphism. The following are equivalent.
\begin{enumerate}[leftmargin=*,label=(\roman*)]
\item The morphism $f$ is Parikh-collinear. 
\item For all $k\geq 0$,  $f$ satisfies $P_k$.
\item There exists an integer $k\ge 0$ such that $f$ satisfies $P_k$.
\end{enumerate}
\end{theorem}
Before proving this result in \cref{sec:thm:proof-Parikh-collinear_bounded-k_to_bounded_k+1}, let us mention a straightforward consequence, which generalizes \cite[Thm.~13]{RigoSalimov2015} from Parikh-constant to
Parikh-collinear morphisms. For example, the Thue--Morse morphism is Parikh-constant and thus Parikh-collinear but the morphism of \cref{exa:pc} is Parikh-collinear but not Parikh-constant.


\begin{corollary}\label{cor:k-binomial_bounded}
Let $\infw{z}$ be a fixed point of a Parikh-collinear morphism.
For any $k \geq 1$ there exists a constant $C_{\infw{z},k} \in \N$ such that
$\bc{\infw{z}}{k}(n)\leq C_{\infw{z},k}$ for all $n \in \N$.
\end{corollary}
\begin{proof}
Let $f\colon A^* \to A^*$ be a Parikh-collinear morphism
whose fixed point is $\infw{z}$. Since $f(\infw{z}) = \infw{z}$,
\cref{thm: abelian complexity iff Parikh-collinear} implies that
$\infw{z}$ has bounded abelian complexity. For any $k\geq 1$, we have
that $\infw{z} = f(f^{k-1}(\infw{z}))$ implying that $\infw{z}$ has
bounded $k$-binomial complexity by induction and the previous theorem.
\end{proof}

\begin{remark}
We cannot relax the (implicit) assumption on the rank of the adjacency matrix~$M_f$ in \cref{cor:k-binomial_bounded}.
For example, the morphism $f: \{0,1,2\}^* \to \{0,1,2\}^*$ defined by
$0\mapsto 0^32^3$, $1\mapsto 0^31^32$, $2\mapsto 2^40^61^3$ has
an adjacency matrix of rank $2$. The fixed point $\infw{x}$
starting with $0$ is aperiodic as $f^n(0)$ is readily seen to be
right special for all $n\geq 0$. Yet, its adjacency matrix has
eigenvalues $\theta_1 = 5 + \sqrt{13}$,
$\theta_2 = 5 - \sqrt{13}$, and $0$, and the former two are 
 greater than $1$. This means that the word has unbounded
abelian complexity. Indeed, a deep result of Adamczewski on
balances in primitive pure morphic words
\cite[Thm.~13(ii)]{Adamczewski2003balances}
implies that the ``$\limsup$''-growth of the function
\[
n \mapsto \max_{a \in \Sigma,u,v\in \Fac_n(\infw{x})}
	\left\{\bigl||u|_a - |v|_a\bigr|\right\}
\]
grows as $\Theta(n^{\log_{\theta_1}\theta_2})$, where
$\log_{\theta_1}\theta_2 \approx 0{,}15448$. It follows (see, e.g., \cite[Lem.~2.2]{RichommeSZ2011Abelian}) that $\bc{\infw{x}}{k}$ is unbounded for each $k\geq 1$.
\end{remark}

\subsection{An Intermediate Characterization of Parikh-Collinearity}\label{sec:charPCM}
To prove \cref{thm:Parikh-collinear_bounded-k_to_bounded_k+1}, we give further
characterizations of Parikh-collinear morphisms. To this end, we require the
following lemma where we define a map $g_e$ which is constant on any abelian
equivalence class. Such a map is natural to consider in view of
\eqref{eq:morphic_image_coeff}.

\begin{lemma}\label{lem:g-function}
Let $A,B$ be finite alphabets with $|A|\ge 2$.
Let $f\colon A^* \to B^*$ be a Parikh-collinear morphism.
For a word $e = e_1 \cdots e_n$ of length $n$ over $B$, define $g_e \colon A^n \to \N$ by
\[
g_e(a_1 \cdots a_n) := \prod_{i=1}^n \binom{f(a_i)}{e_i}.
\]
Then, for all words $w,w'\in A^n$ with $w \sim_1 w'$, we have $g_e(w) = g_e(w')$.
\end{lemma}
\begin{proof}
Write $w = a_1\cdots a_n$ with $a_i\in A$ for all $i\in \{1,\ldots,n\}$.
For all $\alpha\in A$ and $\beta\in B$, define $I(\alpha,\beta) := \left\{ i \in \{1,\ldots,n\} \mid a_i = \alpha \text{ and } e_i = \beta\right\}$.
We get
\[
g_e(w)
= 
\prod_{\substack{\alpha \in A\\ \beta\in B}} \prod_{i\in I(\alpha,\beta)}  \binom{f(\alpha)}{\beta}.
\]
The claim is trivial if $f$ maps all words to $\varepsilon$,
so let $0 \in A$ be a letter for which $|f(0)| \neq 0$. Since the morphism $f$ is Parikh-collinear, for all $\alpha \in A$ and all $\beta\in B$, there exists $r_\alpha \in\mathbb{Q}$ such that $\binom{f(\alpha)}{\beta} = r_\alpha \binom{f(0)}{\beta}$.
We now get
\begin{align*}
g_e(w)
& = 
\prod_{\substack{\alpha \in A\\ \beta\in B}} \prod_{i\in I(\alpha,\beta)}   \binom{f(\alpha)}{\beta}
= 
 \prod_{\substack{\alpha \in A\\ \beta\in B}} \prod_{i\in I(\alpha,\beta)}  r_\alpha \binom{f(0)}{\beta}\\
 &=
\left( \prod_{\substack{\alpha \in A\\ \beta\in B}} \prod_{i\in I(\alpha,\beta)} \binom{f(0)}{\beta} \right) 
\left( \prod_{\substack{\alpha \in A\\ \beta\in B}} \prod_{i\in I(\alpha,\beta)}  r_\alpha \right).
\end{align*}
For any letter $\beta\in B$, the definition of $I(\alpha,\beta)$ gives
\[
\prod_{\alpha\in A} \prod_{i\in I(\alpha,\beta)} \binom{f(0)}{\beta}
=  \binom{f(0)}{\beta}^{|e|_\beta}.
\]
Similarly, for any letter $\alpha \in A$, the definition of $I(\alpha,\beta)$ yields
\[
\prod_{\beta\in B} \prod_{i\in I(\alpha,\beta)}  r_\alpha = r_\alpha^{|w|_\alpha}.
\]
Thus
\[
g_e(w) 
= 
\left( \prod_{\beta\in B} \binom{f(0)}{\beta}^{|e|_\beta} \right) 
\left( \prod_{\alpha \in A} r_\alpha^{|w|_\alpha} \right).
\]
Observe that the first factor in this product only depends on (the Parikh vector of) $e$ --- in particular, not on $w$ --- as the morphism $f$ is fixed. Similarly, the second factor in the product depends solely on the Parikh vector of $w$, not on the word $w$ itself. The desired result follows.
%
\end{proof}

We now characterize Parikh-collinear morphisms by means of binomial complexities.

\begin{proposition}
\label{prop:Parikh-collinear-k-characterization}
Let $f\colon A^*\to B^*$ be a morphism. The following are equivalent.
\begin{enumerate}[leftmargin=*,label=(\roman*)]
\item For all $k\geq 2$ and $u,v \in A^*$, $u \sim_{k-1} v$
implies $f(u) \sim_k f(v)$.
\item There exists an integer $k\geq 2$ such that for all $u,v \in A^*$, $u \sim_{k-1} v$ implies
$f(u) \sim_k f(v)$.
\item For all $u,v \in A^*$, $u \sim_1 v$ implies
$f(u) \sim_2 f(v)$.
\item The morphism $f$ is Parikh-collinear. 
\end{enumerate}
\end{proposition}
\begin{proof}
Clearly $(i)$ implies $(ii)$. We show that $(ii)$ implies $(iii)$. There is nothing to prove if $(ii)$ holds for $k=2$, so assume that $k\geq 3$. We show that
$f$ also satisfies $(ii)$ with $k-1$ instead of $k$, and hence, by repeating
the argument, $f$ satisfies $(ii)$ with $k=2$. Assume to the contrary that there exists a pair $u$, $v$ such that $u\sim_{k-2} v$ but $f(u)\not \sim_{k-1} f(v)$. Since $u$ and $v$ are abelian equivalent ($k-2 \geq 1$) they have equal length, so by \cref{thm:w35}, we have that $uv \sim_{k-1} vu$. Then,
since $f$ has the property for $k$, we have
$f(u)f(v) \sim_{k} f(v)f(u)$. Furthermore, $f(u)$ and $f(v)$ have
the same length (due to $u \sim_1 v$). This implies that $f(u) \sim_{k-1} f(v)$ by the converse part of \cref{thm:w35}, contrary to what was assumed.

Assuming $(iii)$, we show that $(iv)$ holds.
Let $x,y$ be distinct letters from $A$. Since $xy \sim_1 yx$,
we have $f(xy) \sim_2 f(yx)$ by assumption. In other words, for all $s,t \in B$ we have, applying \eqref{eq:morphic_image_coeff},
\begin{align*}
0 &= \binom{f(xy)}{st} - \binom{f(yx)}{st}
=\sum_{\substack{a_1,\ldots,a_{\ell} \in A\\ \ell \leq 2} }
\left[\binom{xy}{a_1\cdots a_{\ell}} - \binom{yx}{a_1\cdots a_{\ell}}\right] 
						\sum_{\substack{st=b_1\cdots b_{\ell}\\b_i\in B^+}}\prod_{i=1}^{\ell}\binom{f(a_i)}{b_i}
\\
&
= \sum_{a_1,a_2 \in A}\left(\binom{xy}{a_1 a_2} - \binom{yx}{a_1 a_2}\right)\binom{f(a_1)}{s}\binom{f(a_2)}{t}
  = \binom{f(x)}{s}\binom{f(y)}{t} - \binom{f(y)}{s}\binom{f(x)}{t},
\end{align*}
where in the third equality we use $\binom{xy}{a} = \binom{yx}{a}$ for all $a \in A$ (since $xy \sim_1 yx$).
Summing over $s\in B$, we get $|f(x)|\binom{f(y)}{t} = |f(y)|\binom{f(x)}{t}$ for all $t \in B$. Now $x$ and $y$ were chosen arbitrarily from the alphabet~$A$.
If $|f(x)|=0$ for all $x\in A$, then $f$ is clearly Parikh-collinear.
If there is a letter $x$ for which $|f(x)| > 0$,
we may write
$\left(\binom{f(y)}{t}\right)_{t\in B} = \frac{|f(y)|}{|f(x)|}\left(\binom{f(x)}{t}\right)_{t \in B}$ for each $y \in A$. In other words, $f$ is Parikh-collinear.

To complete the proof, we show that $(iv)$ implies $(i)$. So let $f$ be a Parikh-collinear morphism and $u\sim_{k-1} v$ with $k\geq 2$. We again apply \eqref{eq:morphic_image_coeff}: for any word $e \in B^*$, we have
\[
\binom{f(u)}{e} - \binom{f(v)}{e}  = \sum_{\substack{a_1,\ldots,a_{\ell} \in A\\ \ell \leq |e|} }\left(\binom{u}{a_1\cdots a_{\ell}} - \binom{v}{a_1\cdots a_{\ell}}\right)
						\sum_{\substack{e=e_1\cdots e_{\ell}\\e_i\in B^+}}\prod_{i=1}^{\ell}\binom{f(a_i)}{e_i}.
\]
Notice that for words $e\in B^{<k}$, we have $\binom{u}{a_1\cdots a_{\ell}} = \binom{v}{a_1\cdots a_{\ell}}$ since $u \sim_{k-1} v$, which in turn gives $\binom{f(u)}{e} = \binom{f(v)}{e}$.
So to show that $f(u) \sim_{k} f(v)$, it suffices to consider words $e \in B^k$.
By assumption, for $\ell < k$, we again have $\binom{u}{a_1\cdots a_{\ell}} = \binom{v}{a_1\cdots a_{\ell}}$. Therefore,
we have $\binom{f(u)}{e} = \binom{f(v)}{e}$
if and only if
\begin{equation}\label{eq:prop_equation}
\sum_{a_1, \ldots, a_{k} \in A}\binom{u}{a_1\cdots a_{k}} \prod_{i=1}^{k}\binom{f(a_i)}{e_i} = \sum_{a_1,\ldots, a_{k} \in A}\binom{v}{a_1\cdots a_{k}} \prod_{i=1}^{k}\binom{f(a_i)}{e_i}.
\end{equation}
Observe here that $\prod_{i=1}^{k}\binom{f(a_i)}{e_i} = g_e(a_1\cdots a_{k})$ as defined in \cref{lem:g-function}.
Let $\mathcal{C}$ be an abelian equivalence class in $A^k/{\sim_1}$. 
By \cref{lem:g-function}, $g_e(\cdot)$ is constant on $\mathcal{C}$, so write $g_e(w)=g_{\mathcal{C},e}$ for all words $w\in\mathcal{C}$. For each $w \in \mathcal{C}$ we may write $\Psi(w)=(m_{\mathcal{C},a})_{a \in A}$.
We now have
\begin{align*}
\sum_{w \in A^{k}}\binom{u}{w}g_e(w) 
&= \sum_{\mathcal{C}\in A^k/{\sim_1}}\sum_{w \in \mathcal{C}}\binom{u}{w}g_e(w)
= \sum_{\mathcal{C}\in A^k/{\sim_1}}\!\!\!g_{\mathcal{C},e}\sum_{w \in \mathcal{C}}\binom{u}{w} \nonumber
= \sum_{\mathcal{C}\in A^k/{\sim_1}}\!\!\! g_{\mathcal{C},e} \prod_{a \in A}\binom{|u|_a}{m_{\mathcal{C},a}},
\end{align*}
where the last equality is from~\cref{lem:sum-constantPvect}.
One obtains the same formula by replacing $u$ with $v$, and equality
indeed holds in \eqref{eq:prop_equation} as $|u|_a = |v|_a$ for each
letter $a\in A$. This concludes the proof.
\end{proof}

\begin{remark}
In \cite[Lem.~5]{ManvelMeyerowitz}, the authors show that, for
a morphism $f$ such that $f(a) \sim_h f(b)$ for all $a,b \in A$,
for all words $u$, $v \in A^*$ with $u \sim_k v$ we have
that $f(u) \sim_{k+h} f(v)$. Towards the converse, assume that
$f$ is a morphism for which the conclusion holds
(for all $k\geq 1$ but fixed $h\geq 1$). Then we necessarily have
$f(a)^m f(b)^n \sim_{h+1} f(b)^n f(a)^m$ for all $a,b\in A$, $m,n\geq 1$. From this we
infer that, e.g., $f(a)^{|f(b)|} \sim_h f(b)^{|f(a)|}$ (as a corollary
of \cref{thm:w35}).
In particular, if~$f$ is uniform, we have $f(a) \sim_h f(b)$ for all $a,b \in A$. It would be interesting
to characterize the non-uniform morphisms $f$ with this property.
For example, one can take any Parikh-collinear morphism $g$; then
$f=g^h$ is such a morphism. We highly suspect that these are not
the only such morphisms.
%
\end{remark}

\subsection{Proof of \texorpdfstring{\cref{thm:Parikh-collinear_bounded-k_to_bounded_k+1}}{}}\label{sec:thm:proof-Parikh-collinear_bounded-k_to_bounded_k+1}



%
%

We require the following technical result, which essentially appears in the proof of \cite[Thm.~12]{CassaigneRichomeSaariZamboni2011}. We give a proof here for the sake of completeness.
\begin{lemma}\label{lem:technical_image_lengths}
Let $\infw{x}$ be a an infinite word over $A$ with bounded abelian complexity.
Let $f: A^* \to B^*$ be a morphism and assume $\infw{y} = f(\infw{x})$ is an infinite word.
Then for all $c \in \N$ there exists $D_{\infw{x},c} \in \N$ such that if $\big||f(u)| - |f(v)|\big| \leq c$, for some $u,v \in \mathcal{L}(\infw{x})$, then $\big||u|-|v|\big| \leq D_{\infw{x},c}$.
\end{lemma}
\begin{proof} Assume without loss of generality that $|u| \geq |v|$ and write $u = u'v'$ with $|v'| = |v|$.
Let $M_f$ be the adjacency matrix of $f$.
If $\big| |f(u)| - |f(v)| \big| \leq c$, we have by the reverse triangle inequality
\[
c \geq \big||f(u')| - |f(v)| + |f(v')|\big|
\geq |f(u')| - \big| |f(v')|-|f(v)| \big| 
 = |f(u')| - |\langle M_f(\Psi(v') - \Psi(v)),\vec{1} \rangle|,
\]
where $\langle \cdot\, , \,\cdot\rangle$ denotes the inner product of vectors, and $\vec{1}$ is the all-ones-vector. Recall that $\infw{x}$ has bounded abelian complexity if and only if it is $C$-balanced for some $C$ \cite{RichommeSZ2011Abelian}.  
Hence, as $v$ and $v'$ are factors of the same length, $\Psi(v') - \Psi(v)$ attains finitely many distinct integer points (in particular, belonging to $[-C,C]^{\# A}$). So does $M_f(\Psi(v') - \Psi(v))$.
We therefore obtain $|f(u')| \leq D$ for some $D\in \N$.
We deduce that $u'$ is bounded in length as well: indeed, let
$a\in A$ be a letter occurring infinitely often in
$\infw{x}$ and for which $f(a) \neq \varepsilon$ (such a letter exists
because $f(\infw{x})$ is infinite). Since
$\infw{x}$ is balanced, we deduce that all long enough factors
of $\infw{x}$ contain more than $|u'|$ occurrences of $a$.
We let $D_{\infw{x},c}$ be this bound on $|u'|$ to conclude the proof.
\end{proof}

We are now ready to prove the main result of this section, characterizing Parikh-collinear morphisms in terms the property $P_k$ defined at the beginning of \cref{sec: rank1}.

\begin{proof}[Proof of \cref{thm:Parikh-collinear_bounded-k_to_bounded_k+1}]
Let us first show that $(i)$ implies $(ii)$.
Assume thus that $f$ is Parikh-collinear. \Cref{thm: abelian complexity iff Parikh-collinear} implies that $f$ maps all words (i.e., all words with bounded $0$-binomial complexity) to words with bounded $1$-binomial complexity. Let $k\geq 1$ and let $\infw{x}$ be a word with bounded $k$-binomial complexity. Let $n\in \N$.
Any length-$n$ factor of $f(\infw{x})$ can be written as
$p f(u) s$, where the word $u$ is a factor of $\infw{x}$, $p$ is a suffix of $f(a)$ and $s$ is a prefix of $f(b)$ for some letters
$a,b\in A$.
Here $n - 2m < |f(u)| \leq n$,
where $m:=\max_{a\in A}|f(a)|$.
The $(k+1)$-binomial equivalence class of $p f(u) s$ is completely
determined by the words $p$, $s$, and the $k$-binomial equivalence
class of $f(u)$, which itself is determined by the
abelian equivalence class of $u$ by
\cref{prop:Parikh-collinear-k-characterization}.

The former two words $p$ and $s$ are drawn from a finite set, as their lengths are bounded by the constant $m$ (depending on $f$).
The length of $u$ can be chosen from an interval whose length is uniformly bounded in $n$. Indeed, assume we have equal length factors $w = p f(u) s$ and $w' = p' f(v) s'$. As observed above,
$n \geq |f(u)|$ and $|f(v)| > n-2m$, so that
$\big| |f(u)| - |f(v)| \big| < 2m$.
Applying \cref{lem:technical_image_lengths} (by assumption, $\infw{x}$ has bounded $k$-binomial complexity and thus, $\infw{x}$ has bounded
abelian complexity by \eqref{eq:ineq}) there exists a bound $D$
such that $\big||u| - |v|\big| \leq D$ uniformly in $n$.
Since the number of $k$-binomial equivalence classes in $\infw{x}$ of each length is uniformly bounded by assumption, and the number of admissible lengths for $u$ above is bounded, we conclude that the number of choices for the $k$-binomial equivalence class of $u$ is bounded. We have shown that the number of
$(k+1)$-binomial equivalence classes among factors of length $n$ in $f(\infw{x})$ is determined from a bounded amount of information
(not depending on $n$), as was to be shown.
Consequently,  $f$ satisfies $P_k$.

Notice that $(ii)$ trivially implies $(iii)$.

Let us turn to the last implication, namely $(iii)$ implies $(i)$.
Assume $(iii)$ holds,  that is,  for some integer $k\ge 0$,  $f$ satisfies $P_k$.
If $k=0$,  then $f$ maps all words to words with bounded $1$-binomial complexity,  so $f$ is Parikh-collinear by \cref{thm: abelian complexity iff Parikh-collinear}.
Assume that $k\ge 1$, and towards a contradiction, assume further that $f$ is not Parikh-collinear.
By \cref{prop:Parikh-collinear-k-characterization}, there exist words $u,v$ with $u \sim_k v$ and $f(u) \not\sim_{k+1} f(v)$.
Write $U=f(u)$ and $V=f(v)$.
Now define the word $\infw{x}=uvu^2v^2u^3v^3 \cdots u^n v^n \cdots$ and consider
\[
f(\infw{x})=UVU^2V^2U^3V^3 \cdots U^n V^n \cdots.
\]
Below we show that $\infw{x}$ has bounded $k$-binomial complexity, while $f(\infw{x})$ has unbounded $(k+1)$-binomial complexity, which is enough to contradict $(iii)$.

Since $u$ and $v$ are $k$-binomially equivalent, $\bc{\infw{x}}{k}$ is bounded. 
(To see this,  one may apply arguments similar to those developed in the first part of the proof.)
Let us prove the second.
For each integer $n$, $f(\infw{x})$ contains the factors $U^r V^{n-r}$ with $r\in\{0,\ldots,n\}$.
These factors are actually all $(k+1)$-binomially inequivalent.
Indeed, assume towards a contradiction that $U^rV^{n-r} \sim_{k+1} U^sV^{n-s}$ for some $r,s\in\{0,\ldots,n\}$ with $r>s$.
By the Cancellation property (\cref{lem: cancellation property}), we obtain $U^{r-s} \sim_{k+1} V^{r-s}$.
\cref{lem:diff-powers} then implies that $U \sim_{k+1} V$, which is a contradiction.
Consequently, $\bc{f(\infw{x})}{k+1}$ is unbounded, as desired.
\end{proof}

\section{Binomial Properties of the Thue--Morse Morphism, Part I}
\label{sec:propsTMI}
\label{sec: answer question stab}

In this section, we consider binomial complexities of iterates of the Thue--Morse morphism $\varphi$ on aperiodic binary words.
The section is split into three subsections. To state the main result, we define the following.
\begin{definition}
Let $\infw{x}$ be a binary word and $k \geq 1$ an integer.
We say that $\infw{x}$ has property $\tmprop(k)$ if, for all $1 \leq j \leq k$, we have $\bc{\infw{x}}{j} = \bc{\infw{t}}{j}$.
\end{definition}
Recall that the exact values for $\bc{\infw{t}}{j}$
were computed in \cite[Thm.~6]{LejeuneLeroyRigo2020} (and are
given by \cref{eq:k-image-j-complexities}).
The main result of~\cref{sec:first-k-TM} is the following
theorem, which can be seen as a generalization of the aforementioned result.

\begin{theorem}\label{thm:k-image-j-complexities}
Let $k$ be an integer and let $\infw{y}$ be an aperiodic binary
word. Then the word $\infw{x} = \varphi^k(\infw{y})$ (and any of its suffixes) has property $\tmprop(k)$.
\end{theorem}
The application of $\varphi^k$ to a word changes the $j$-binomial
complexities, $j \leq k$, to that of the Thue--Morse word's.
Putting this bluntly, the binomial complexities of the original
word play no role in the $j$-binomial complexities of the image word (for small $j$).

The topic of \cref{sec:morebin} is to characterize the $k$- and $(k+1)$-binomial
equivalence among factors of words of the form $\varphi^k(\infw{y})$
(\cref{thm:k-image-prefix-suffix-relation,pro:k-image_k+1-prefix-suffix}). In the
latter, we see that structure of $\infw{y}$ already appears to affect the $(k+1)$-binomial complexity of
$\varphi^k(\infw{y})$. This allows to conclude, for example, that $\bc{}{k} \prec \bc{}{k+1}$
(\cref{cor:k-image_k+1-prec}) for words of this form. Throughout the rest of this section we fix $\infw{x}$ and $\infw{y}$ to be as in \cref{thm:k-image-j-complexities}.

We begin with a subsection introducing a convenient tool, called abelian Rauzy graphs,
which we use throughout the current and
the following section.

\subsection{Abelian Rauzy graphs}\label{sec:abelian-Rauzy}
For an infinite word $\infw{z} \in A^{\N}$, consider a sliding
window of length $n$: as the window shifts, one goes from heavier
factors to lighter factors and vice versa.
One can consider a directed labeled graph $G = (V,E)$ capturing its progress: the
vertices are the Parikh vectors of factors of length $n$, and
there is an edge from $\vec{x}$ to $\vec{y}$ labeled with $(a,b) \in A \times A$, if there exists
$aub \in \Fac_{n+1}(\infw{z})$ such that $\Psi(au) = \vec{x}$
and $\Psi(ub) = \vec{y}$. We call $G$ the \emph{abelian Rauzy graph (of order $n$)}. Such graphs were considered already in
\cite{RichommeSZ2010balance}.

\begin{remark}
  The abelian Rauzy graph $G=(V,E)$ defined here is a quotient of the usual Rauzy graph of $\infw{z}$ of order $n$. The latter one is defined as $R=(V',E')$ where $V'=\Fac_{n}(\infw{z})$ and there is an edge from $au$ to $ub$ of label $(a,b)$ whenever $aub\in \Fac_{n+1}(\infw{z})$. The Parikh map $\Psi:V'\to V$ is a morphism of graphs: any labeled path in $R$ is mapped to a path in $G$ with same label.
\end{remark}

We describe some properties of abelian Rauzy graphs.
\begin{observation}\label{obs:graph-structure}
Let $G = (V,E)$ be the abelian Rauzy graph of order $n$ of an infinite word $\infw{z} \in A^{\N}$.
\begin{itemize}
\item The number of vertices is $\# V=\bc{\infw{z}}{1}(n)$;
  
\item An edge with label $(a,b)$ corresponds to an increase in weight if and only if $ab = 01$;

\item An edge with label $(a,b)$ corresponds to a decrease in weight if and only if $ab = 10$;

\item An edge is a loop if and only if $a = b$.

\item A right special factor of length $n$ gives rise to a \emph{right special vertex}:
a vertex with two outgoing edges for which the labels have the same first component.
If $\infw{z}$ is binary, then one of the two edges is a loop.

\item Each vertex has at least one outgoing edge (possibly a loop).
Furthermore, if the word $\infw{z}$ is aperiodic, each vertex has at least one
outgoing edge that is not a loop, and there is at least one vertex with two outgoing edges.
In particular, the graph has at least $\#V+1$ edges.
\end{itemize}
\end{observation}

\begin{example}\label{ex:abelian-Rauzy-TM}
Let us consider the abelian Rauzy graphs of the Thue--Morse word.
For a fixed $n$, we identify the vertices of $G$, that is, the
Parikh vectors of factors of length $n$, with their second
components. Indeed,  for a binary word $u$ of a fixed length $n$,
we have $\Psi(u) = (n-|u|_1,|u|_1)^{\intercal}$. We show that,
for all $m\geq 1$,  $G_{2m}$ and $G_{2m+1}$ take the following forms:
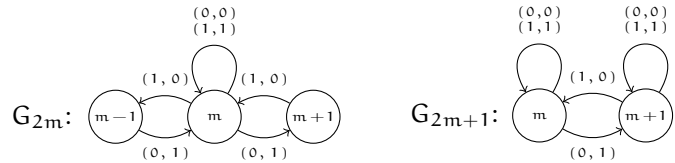
\begin{center}
\begin{tikzpicture}[x=1.3cm]
\node at (.25,0) {$G_{2m}$:};
\node[circle,draw,minimum size = 20pt,inner sep = 0pt] (x2) at (1,0) {\tiny $m\!-\!1$};
\node[circle,draw,minimum size = 20pt,inner sep = 0pt] (x3) at (2,0) {\tiny $m$};
\node[circle,draw,minimum size = 20pt,inner sep = 0pt] (x4) at (3,0) {\tiny $m\! +\! 1$};

\draw[->] (x2) edge[out=330,in=210] node[midway,below] {\tiny $(0,1)$} (x3);
\draw[->] (x3) edge[in=30,out=150] node[midway,above] {\tiny $(1,0)$} (x2);

\draw[->] (x3) edge[out=330,in=210] node[midway,below] {\tiny $(0,1)$} (x4);
\draw[->] (x4) edge[in=30,out=150] node[midway,above] {\tiny $(1,0)$} (x3);

\draw[->] (x3) edge[in=120,out=60,looseness=7] node[midway,above] {\tiny $\substack{(0,0)\\(1,1)}$} (x3);
\end{tikzpicture}
\qquad
\begin{tikzpicture}[x=1.4cm]
\node at (.2,0) {$G_{2m+1}$:};
\node[circle,draw,minimum size = 20pt,inner sep = 0pt] (x2) at (1,0) {\tiny $m$};
\node[circle,draw,minimum size = 20pt,inner sep = 0pt] (x3) at (2,0) {\tiny $m\!+\!1$};

\draw[->] (x2) edge[out=330,in=210] node[midway,below] {\tiny $(0,1)$} (x3);
\draw[->] (x3) edge[in=30,out=150] node[midway,above] {\tiny $(1,0)$} (x2);

\draw[->] (x2) edge[in=120,out=60,looseness=7] node[midway,above] {\tiny $\substack{(0,0)\\(1,1)}$} (x2);
\draw[->] (x3) edge[in=120,out=60,looseness=7] node[midway,above] {\tiny $\substack{(0,0)\\(1,1)}$} (x3);
\end{tikzpicture}
\end{center}
Let us write $\infw{t} = 01101001\cdots = a_0a_1a_2\cdots$.
We first consider $G_{2m+1}$. It is well-known and plain to see that $\infw{t}$
is closed under complementation; for all $u \in \Fac(\infw{t})$,
we have $\overline{u} \in \Fac(\infw{t})$. Let $au$, $a \in \{0,1\}$,
$au \in \Fac_m(\infw{t})$, be right special in $\infw{t}$.
Consequently, $\varphi(au)a \in \Fac(\infw{t})$; hence $G_{2m}$
contains the loop $m \xrightarrow{(a,a)} m$. By complementing such a
factor, we also find $m \xrightarrow{(\overline{a},\overline{a})} m$ in $G_{2m+1}$.
To see that, e.g., there is no loop at $m-1$, we note that any $m-1$-factor is
of the form $0\varphi(u')0$, and appears at an odd position in $\infw{t}$.
Hence there is certainly no loop with label $(1,1)$ at $m-1$. Neither can
there be a loop with label $(0,0)$, as $00$ is not the image
of a letter.

We then inspect $G_{2m+1}$. The following statements are easy to prove using,
e.g., the automatic prover \texttt{Walnut} \cite{Mousavi2016automatic}. For a comprehensive take
on the usage of \texttt{Walnut}, we recommend the book \cite{Shallit2022logical}.
\begin{itemize}
\item For all $m\geq 1$ there exists a length-$(2m+2)$ factor $u$ starting at an even index in
$\infw{t}$, and which begins and ends with $1$.\\
The following \texttt{Walnut} formula returns 'TRUE':\\
\verb|eval OddL11 "Am (m>0) => Ej T[2*j]=@1 & T[2*j+2*m+1]=@1";|

\item For all $m \geq 0$ there exists a length-$(2m+2)$ factor starting at an odd index in
$\infw{t}$, and which begins and ends with $0$.\\
The following \texttt{Walnut} formula returns 'TRUE':\\
\verb|eval OddL00 "Am Ej j>0 & T[2*j-1]=@0 & T[2*j+2*m]=@0";|
\end{itemize}
Then the first item implies that $G_{2m+1}$ contains the loop $m \xrightarrow{(1,1)} m$.
Indeed, the length-$(2m+1)$ prefix of $u$ in the first item has weight $m$
(because $u = \varphi(u')01$ for some $|u'|=m$). Similarly the second item implies
that $m$ has a loop with label $(0,0)$. Recalling that $\infw{t}$ is closed under complementation,
$G_{2m+1}$ is seen to be of the claimed form.
\end{example}

We show the structure of the abelian Rauzy graphs of Sturmian words in \cref{prop:sturmian-abelian-Rauzy-graphs}.

We shall make use of the following general lemma, a part of which
appears as \cite[Lem.~3.2.]{RichommeSZ2011Abelian}
\begin{lemma}\label{lem:boundaries}
Let $\infw{z}$ be an aperiodic binary word. Then, for every
$n \geq 1$, the set of edge labels of the abelian Rauzy graph $G_n$ contains
$(0,1)$, $(1,0)$, and either $(0,0)$ or $(1,1)$. Furthermore, if $(a,a)$ does
not appear as a label of a loop in $G_n$, then $G_{n+1}$ contains a loop
with the label $(\overline{a},\overline{a})$.
\end{lemma}

\begin{proof}
Since $\infw{z}$ is aperiodic, it must have a connected component containing a
loop (a right special vertex) and at least two vertices (by a theorem of Coven and Hedlund \cite{CovenH1973sequences}).
An edge from a lighter vertex to a heavier one has label $(0,1)$, and $(1,0)$ appears as the
label from the heavier to the lighter one.

Assume that $G_n$ does not contain a loop with label $(0,0)$.
Consider the walk $W$ in $G_n$ defined by $\infw{z}$: in particular,
consider the strongly connected subgraph of $G_n$ comprising the
vertices and edges that $W$ traverses infinitely many times.
There is an edge from the second lightest vertex to the lightest
one {\tt l}, labeled with $(1,0)$. This means that the factor of length
$n+1$ corresponding to this edge begins with $1$ and ends with $0$.
Since $(0,0)$ does not appear as a label in $G_n$ and from the lightest vertex there is no outgoing edge with label $(1,0)$, the edge that
$W$ takes from {\tt l} has label $(\cdot,1)$. Thus the factor of length $n+2$ begins and ends with $1$.
This gives an edge in $G_{n+1}$ with label $(1,1)$.
\end{proof}
%

\subsection{The First \texorpdfstring{$k$}{k} Binomial Complexities}\label{sec:first-k-TM}



We begin by defining the notion of $\varphi^j$-factorizations of factors of $\infw{x}$. This will be used throughout this and the next section.

\begin{definition}\label{def:varphij-factorization}
For any factor $u$ of $\varphi^j(\infw{y})$ of length at least $2^j-1$ there exist $a,b \in \{0,1\}$ and $z \in \{0,1\}^*$ with $azb\in \mathcal{L}(\infw{y})$ such that
$u = p\varphi^j(z)s$ for some
proper suffix $p$ of $\varphi^j(a)$ and some proper prefix $s$ of
$\varphi^j(b)$. (Note that $z$ could be empty.) The triple $(p,\varphi^j(z),s)$ is called a {\em $\varphi^j$-factorization}\footnote{We warn the reader that the term $\varphi$-factorization has a different meaning in~\cite{LejeuneLeroyRigo2020}. Our $\varphi^j$-factorization corresponds to their ``factorization of order $j$''.} of $u$. The word $azb$ (resp., $zb$; $az$; $z$) is said to be the corresponding {\em $\varphi^j$-ancestor} of $u$ when $p,s$ are non-empty (resp., $p=\varepsilon$ and $s\neq\varepsilon$; $p\neq\varepsilon$ and $s=\varepsilon$; $p=s=\varepsilon$).
\end{definition}

Since the words $\varphi^j(0)$ and $\varphi^j(1)$
begin with different letters, we notice that if $s \neq \varepsilon$ in a $\varphi^j$-factorization of a word,
then the letter $b$ is uniquely determined. Similarly the $j$th images
of the letters end with distinct letters (for $j$ fixed), whence the letter $a$ is uniquely determined once $p \neq \varepsilon$.

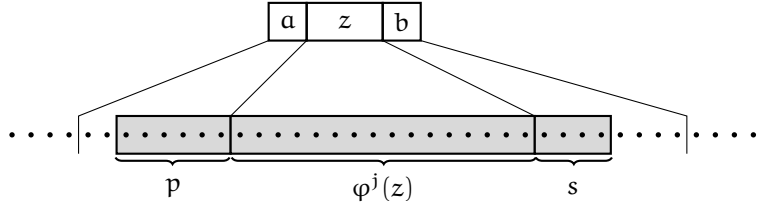
\begin{figure}[h!t]
  \begin{center}
    \begin{tikzpicture}
      \node [fit={(5,2.5) (5.5,3)}, inner sep=0pt, draw=black, thick] (a) {};
       \node [fit={(5.5,2.5) (6.5,3)}, inner sep=0pt, draw=black, thick] (y) {};
       \node [fit={(6.5,2.5) (7,3)}, inner sep=0pt, draw=black, thick] (b) {};
  \node [fit={(3,1) (4.5,1.5)}, inner sep=0pt, draw=black, thick, fill=gray!30] (u) {};
  \node [fit={(4.5,1) (8.5,1.5)}, inner sep=0pt, draw=black, thick, fill=gray!30] (x) {};
  \node [fit={(8.5,1) (9.5,1.5)}, inner sep=0pt, draw=black, thick, fill=gray!30] (v) {};
  \draw [thick, decoration={brace, mirror, raise=0.3cm}, decorate] (u.west) -- (x) {};
  \draw [thick, decoration={brace, mirror, raise=0.3cm}, decorate] (x.east) -- (v.east) ;
  \draw [thick, decoration={brace, mirror, raise=0.3cm}, decorate] (u.east) -- (x.east) ;
  \node at (3.75,.55) {$p$};
  \node at (6.5,.55) {$\varphi^j(z)$};
  \node at (9,.55) {$s$};
  \node at (5.25,2.75) {$a$};
  \node at (6,2.75) {$z$};
  \node at (6.75,2.75) {$b$};
  \draw (5.5,2.5) -- (4.5,1.5);
  \draw (5,2.5) -- (2.5,1.5) -- (2.5,1) ;
  \draw (6.5,2.5) -- (8.5,1.5);
   \draw (7,2.5) -- (10.5,1.5) -- (10.5,1) ;
  \foreach \r in {6,...,45}{
    \node[circle,fill=black,inner sep=0pt,minimum size=2pt] (a) at (.125+.25*\r,1.25) {};
    }
\end{tikzpicture}
\end{center}
\caption{A $\varphi^j$-factorization and its $\varphi^j$-ancestor.}
  \label{fig:frame}
\end{figure}

The following lemma says that an aperiodic word of the form
$\varphi^k(\infw{y})$ has the same short factors as the Thue--Morse word.
\begin{lemma}\label{lem:k-image-short-factors}
For an aperiodic binary word $\infw{y}$ and integer $k$,
we have
$\mathcal{L}_{n}(\varphi^k(\infw{y})) = \mathcal{L}_{n}(\infw{t})$ for all $n \leq 2^k$.
\end{lemma}
\begin{proof}
Let $\infw{x} = \varphi^k(\infw{y})$. The claim is trivial for $k=0$ so assume $k\geq 1$.
Any factor of $\infw{x}$ of length at most $2^k$ appears as
a factor of $\varphi^k(v)$, where $v\in \Fac_2(\infw{y})$.
Since $\Fac_2(\infw{t}) = \{0,1\}^2$, all such factors appear
in $\varphi^k(\infw{t})=\infw{t}$. In particular we have shown $\mathcal{L}_{n}(\infw{x}) \subseteq \mathcal{L}_{n}(\infw{t})$
for $n \leq 2^k$.

Since $\infw{y}$ is aperiodic it contains both $01$ and
$10$ and either of $00$, and $11$. If $\infw{y}$ contains all
factors of length $2$, then clearly the considered languages are
equal. So assume without loss of generality that $11$ does not
appear in $\infw{y}$. Consider the factors of length at most
$2^k$ of
$\varphi^k(11) = \varphi^{k-1}(1)\varphi^k(0)\varphi^{k-1}(0)$.
Such a factor is either a factor of
$\varphi^{k-1}(1)\varphi^k(0)$ or
$\varphi^{k}(0)\varphi^{k-1}(0)$, both of which are factors
of $\varphi^k(00)$. Hence all factors of length at most $2^k$
appearing in $\infw{t}$ appear in $\infw{x}$ as well; this
concludes the proof.
\end{proof}

We are in the position to prove the main result of this subsection.

\begin{proof}[Proof of \cref{thm:k-image-j-complexities}]
Let $j\in\{1,\ldots,k\}$. We first prove the claim for $\infw{x}$ (and afterwards the claim for any of its suffixes).
As the factors of length at most
$2^j$ of $\infw{x}=\varphi^j(\varphi^{k-j}(\infw{y}))$
coincide with those of $\infw{t}$ by the above lemma,
the $j$-binomial complexity of $\infw{x}$ coincides with
that of the Thue--Morse word's for $n < 2^j$.

In the remaining of the proof we let $n\geq 2^j$. We show that
$\Fac_n(\infw{t})/{\sim_j} = \Fac_n(\infw{x})/{\sim_j}$ by double inclusion, which
suffices for the claim since \cref{thm:k-image-j-complexities} holds true for $\infw{x} = \infw{t}$.

Let $u \in \mathcal{L}(\infw{x})$; we show that there exists
$v \in \mathcal{L}(\infw{t})$ such that $u \sim_j v$. To this end, let $\infw{z} = \varphi^{k-j}(\infw{y})$ so that
$\infw{x} = \varphi^{j}(\infw{z})$.
Let $u$ have $\varphi^j$-factorization $p\varphi^j(u')s$ with $\varphi^j$-ancestor $au'b\in \Fac(\infw{z})$.
The Thue--Morse word contains a factor $a v' b$, where $|v'| = |u'|$ (see, e.g., \cite[Prop.~33]{LejeuneLeroyRigo2020} or the abelian Rauzy graphs in \cref{ex:abelian-Rauzy-TM}). It follows that $\infw{t}$ contains the factor
$v:=p\varphi^j(v')s$. Now $u \sim_j v$ because $\varphi^{j}(u') \sim_j \varphi^j(v')$ by \cref{thm:Ochsenschlager}.

Let then $u \in \mathcal{L}(\infw{t})$ have $\varphi^j$-factorization $p\varphi^j(u')s$ with $\varphi^j$-ancestor $au'b\in \Fac(\infw{t})$.
As before we show that
there exists $v \in \mathcal{L}(\infw{x})$ such that $u \sim_j v$.
As a consequence of \cref{lem:boundaries}, $\infw{z}$ contains, at
each length, factors from both the languages $0A^*1$ and $1A^*0$.
Hence, if $a$ and $b$ above are distinct, we may argue as in the
previous paragraph to obtain the desired conclusion. Assume thus
that $a = b$. Again \cref{lem:boundaries} implies that $\infw{z}$
contains a factor of length $|u'|+2$ in the language $1A^*1 \cup 0A^*0$.
Assume without loss of generality that it contains a factor from $0A^*0$.
Then, if $a = b = 0$,
we may again argue as in the previous paragraph. So assume now that
$a = b = 1$ and $\Fac_{|u|'+2}(\infw{z}) \cap 1A^*1 = \emptyset$. Notice that \cref{lem:boundaries} implies that
$\Fac_{|u|'+2}(\infw{z}) \cap 0A^*0 \neq \emptyset$ and, further,
$\Fac_{|u|'+2 \pm 1}(\infw{z}) \cap 0A^*0 \neq \emptyset$. 
To conclude with the claim for $\infw{x}$, we have four cases to consider depending on the length of $p$ and $s$ which can be less or equal, or greater than $2^{j-1}$.

\begin{enumerate}[wide=15pt,widest=99,label=\textbf{Case \arabic*:},itemsep=3pt]

\item Assume that $p$ is a suffix of $\varphi^{j-1}(0)$ and $s$ is a prefix of $\varphi^{j-1}(1)$. For all $v'$ such that $|v'| = |u'|-1$, $\varphi^j(u')\sim_j\varphi^j(v'1)$ by \cref{thm:Ochsenschlager}. By the
\hyperref[lem:transfer]{Transfer Lemma} (\cref{lem:transfer}), $\varphi^j(v'1)\sim_j \varphi^{j-1}(1)\varphi^j(v')\varphi^{j-1}(0)$. Consequently
\[
u \sim_j p\varphi^{j-1}(1)\varphi^j(v')\varphi^{j-1}(0)s=:v,
\]
where $p\varphi^{j-1}(1)$ is a suffix of $\varphi^j(0)$ and $\varphi^{j-1}(0)s$ is a prefix of $\varphi^j(0)$. Hence $v$ is a factor of $\varphi^j(0v'0)$.
Recall that a factor of the form $0v'0$ appears in $\infw{z}$ by assumption, and thus $\varphi^{j}(0v'0)$ appears in $\infw{x}$.
To recap, we have shown a factor $v$ of
$\infw{x}$ $j$-binomially equivalent to $u$.

\item Assume that $p=p'\varphi^{j-1}(0)$ where $p'$ is a suffix of $\varphi^{j-1}(1)$ and $s$ is a prefix of
$\varphi^{j-1}(1)$. For all $v'$ such that $|u'|=|v'|$, applying \cref{thm:Ochsenschlager} and \cref{lem:transfer},
$$u\sim_j p' \varphi^j(v')\varphi^{j-1}(0)s=:v.$$
Hence $v$ is a factor of $\varphi^j(0v'0)$, and such a factor appears in $\infw{z}$ by assumption. We conclude as above.

\item Assume that $p$ is a suffix of
$\varphi^{j-1}(0)$ and $s=\varphi^{j-1}(1)s'$ where $s'$ is a prefix of
$\varphi^{j-1}(0)$. For all $v'$ such that $|u'|=|v'|$, applying
\cref{thm:Ochsenschlager,lem:transfer},
we have that $u\sim_j p \varphi^{j-1}(1) \varphi^j(v')s'=:v$
and the conclusion is the same as in the previous case.

\item Assume that $p=p'\varphi^{j-1}(0)$ and $s=\varphi^{j-1}(1)s'$ where $p'$ is a suffix of $\varphi^{j-1}(1)$  and $s'$ is a prefix of $\varphi^{j-1}(0)$.  For all $v'$ such that $|v'|=|u'|+1$, applying \cref{thm:Ochsenschlager} and \cref{lem:transfer},
\[
u\sim_j p' \varphi^{j-1}(0) \varphi^{j-1}(1)\varphi^j(u')s'\sim_j p' \varphi^j(w')s'=:v,
\]
Hence $v$ is a factor of $\varphi^j(0w'0)$ and the conclusion is similar to Case 1.
\end{enumerate}
To conclude the proof, we consider the case of a suffix
$\infw{w}$ of $\infw{x}$. Now $\infw{w}$ has a suffix of the
form $\varphi^k(\infw{y}')$, where $\infw{y}'$ is a suffix of $\infw{y}$. Notice now that $\Fac(\infw{x}) \supseteq \Fac(\infw{w}) \supseteq \Fac(\varphi^k(\infw{y}'))$. The theorem
applies to both $\infw{x}$ and $\varphi^k(\infw{y}')$ by the previous part. Hence $\infw{w}$ has property $\tmprop(k)$ also.
\end{proof}

%
%
%
%

\begin{remark}
If $\infw{y}$ is an aperiodic infinite word, then for all $a,b\in\{0,1\}$ and $n \geq 2$ we have
$\Fac_n(\varphi(\infw{y})) \cap aA^* b \neq \emptyset$.
Indeed, for $a \neq b$ the claim follows from \cref{lem:boundaries}. For $a = b$, we observe the following:
for even length factors $n = 2\ell$, $\ell \geq 1$, a factor $\overline{a} y a$ of $\infw{y}$ of length $\ell + 1$ (which exists by \cref{lem:boundaries}) gives
a factor $\overline{a}a \varphi(y) a\overline{a}$ in $\infw{z}$,
hence we have the factor $a z a$ with $|z| = 2\ell - 2$.
For odd length factors $n  = 2\ell+1$, $\ell \geq 1$, we have
that a factor of the form $c y c$,
$|y| = \ell-1$, of $\infw{y}$
(such a factor exists for some $c \in \{0,1\}$ by \cref{lem:boundaries})
gives $c\overline{c} \varphi(y) c \overline{c}$.
Consequently $\infw{z}$ contains a factor in $a A^* a$ of length $n$ as well.

Applying this observation in the above proof to $\infw{z}$ when
$j<k$, we have $\Fac_{n}(\infw{z}) \cap 1A^*1 \neq \emptyset$
for all $n\geq 2$, and thus case analysis at the end of the proof
is only necessary for the case $j=k$.
\end{remark}

\subsection{On the \texorpdfstring{$k$- and $(k+1)$}{k- and (k+1)}-Binomial Equivalence}\label{sec:morebin}

The previous subsection was dealing with the $j$-binomial equivalence in $\infw{x}=\varphi^{k}(\infw{y})$, where $\infw{y}$ is an aperiodic binary word and $j\leq k$. Here, we are concerned with the $(k+1)$-binomial equivalence in such words. To this end, we need to have more control on the $k$-binomial equivalence in $\infw{x}$. First, we have a closer look at the $\varphi^j$-factorizations of a word and in particular at the associated prefixes and suffixes.


\begin{definition}[{\cite[Def.~43]{LejeuneLeroyRigo2020}}]
\label{def:type}
Let $j\geq 1$. 
 Let us define the equivalence relation $\equiv_j$ on ${A^{< 2^j} \times A^{< 2^j}}$ by $(p_1,s_1) \equiv_j (p_2,s_2)$ whenever there exists $a \in A$ such that one of the following situations occurs:
\begin{enumerate}[itemsep=5pt]
\item\label{it:same_sum}
$|p_1|+|s_1| = |p_2|+|s_2|$ and
\begin{enumerate}[topsep=3pt,itemsep=5pt]
\item 
$(p_1,s_1) = (p_2,s_2)$;
\item
$(p_1, \varphi^{j-1}(a) s_1) = (p_2 \varphi^{j-1}(a), s_2)$;
\item
$(p_2, \varphi^{j-1}(a) s_2) = (p_1 \varphi^{j-1}(a), s_1)$; 
\item
$(p_1,s_1) = (s_2,p_2) = (\varphi^{j-1}(a),\varphi^{j-1}(\overline{a}))$; 
\end{enumerate}
\item \label{it:not_same_sum}
$\bigl| |p_1|+|s_1| - (|p_2|+|s_2|) \bigr| = 2^j$ and
\begin{enumerate}[itemsep=5pt,topsep=3pt]
\item
$(p_1,s_1) = (p_2\varphi^{j-1}(a),\varphi^{j-1}(\bar{a})s_2)$; 
\item
$(p_2,s_2) = (p_1\varphi^{j-1}(a),\varphi^{j-1}(\bar{a})s_1)$.
\end{enumerate}
\end{enumerate}
\end{definition}

The next lemma is essentially \cite[Lem.~40 and 41]{LejeuneLeroyRigo2020} (except that with an arbitrary word $\infw{y}$ instead of the Thue--Morse word~$\infw{t}$, we cannot use the fact that $\infw{t}$ is overlap-free, so factors such as $10101$ may appear in $\infw{y}$). To each $\varphi^j$-factorization there is a natural corresponding $\varphi^{j-1}$-factorization, though two $\varphi^j$-factorizations may correspond to the same $\varphi^{j-1}$-factorization. The next lemma also describes how such factorizations are related.

\begin{lemma}\label{lem:unique_image}
Let $j\ge 1$. Let $u$ be a factor of $\varphi^j(\infw{y})$ such
that $|u|\ge 2^j-1$. Then $u$ has at most two
$\varphi^j$-factorizations. Let further $u$ have a
$\varphi^j$-factorization of the form $(p,\varphi^{j}(z),s)$ and
$z_0 z z_{n+1}$ being the corresponding $\varphi^j$-ancestor
(where according to \cref{def:varphij-factorization}
$z_0,z_{n+1}$ or $z$ could be empty). The factor $u$ has a
unique $\varphi^j$-factorization if and only if the word
$z_0 z z_{n+1}$ contains both letters $0$ and $1$.  Moreover, if
there is another $\varphi^j$-factorization
$(p',\varphi^{j}(z'),s')$ with $\varphi^j$-ancestor
$z_0' z' z_{m+1}'$, then $(p,s) \equiv_j (p',s')$ with
$\big||p|-|p'|\big|=\big||s|-|s'|\big| = 2^{j-1}$,
$z_0 z z_{n+1} = a^{n+2}$, and $z_0' z' z_{m+1}' = \overline{a}^{m+2}$ for some $a \in \{0,1\}$.
\end{lemma}
Otherwise stated, the $\varphi^j$-factorization is not unique if
and only if $u$ is a factor of $\varphi^{j-1}(x)$ with
$x\in (01)^* \cup (10)^* \cup 1(01)^* \cup 0(10)^*$.

\begin{proof}
Since $|u|\ge 2^j-1$, $u$ has a factor of the form
$\varphi^{j-1}(a)$ and thus at least one
$\varphi^j$-factorization of the prescribed form exists with
$z=z_1\cdots z_n$ and $n\ge 0$ ($n=0$ if $z=\varepsilon$).
  %

We first prove the claim for uniqueness by induction on $j$. For
$j=1$, assume that
$u=z_0 \varphi(z_1)\cdots \varphi(z_n)z_{n+1}$ with
$z_0,z_{n+1}\in\{0,1,\varepsilon\}$.
Suppose, as in the statement, that both letters $0$ and $1$ occur in
$z_0\cdots z_{n+1}$.
Then we have $z_iz_{i+1}=01$ (or similarly $10$)
for some $i$. This means that $u$ contains the factor $11$
forcing uniqueness of this kind of a factorization:
$11\not\in\{\varphi(0),\varphi(1)\}$. Assume that the property
holds true up to $j-1$ and prove it for $j\ge 2$. Let
$u=p\varphi^{j}(z_1)\cdots \varphi^{j}(z_n)s$ be a 
$\varphi^j$-factorization and assume that $z_iz_{i+1}=01$ for
some $i$. To this factorization, we have a corresponding
factorization of the form
  $$u=p\varphi^{j-1}(z_1)\varphi^{j-1}(\overline{z_1})\cdots \varphi^{j-1}(z_n)\varphi^{j-1}(\overline{z_n})s.$$
Notice that $p$ is a suffix of $\varphi^{j-1}(\overline{z_0})$
if $|p|<2^{j-1}$ and otherwise,
$p=p'\varphi^{j-1}(\overline{z_0})$ with $p'$ a suffix of
$\varphi^{j-1}(z_0)$. Similarly, $s$ is a prefix of
$\varphi^{j-1}(z_{n+1})$ if $|s|<2^{j-1}$ and otherwise,
$s=\varphi^{j-1}(z_{n+1})s'$ with $s'$ a prefix of
$\varphi^{j-1}(\overline{z_{n+1}})$. Observe that
$z_i\overline{z_i}z_{i+1}\overline{z_{i+1}}=0110$. So by the
induction hypothesis, the $\varphi^{j-1}$-factorization of $u$
is unique. There are at most two $\varphi^{j}$-factorizations corresponding to a
$\varphi^{j-1}$-factorization. But since
$\varphi^{j-1}(1)\varphi^{j-1}(1) \notin\{\varphi^j(0),\varphi^j(1)\}$, the claimed uniqueness follows.

We then prove the claim for non-unique factorizations.
  Assume that $z_0=z_1=\cdots=z_{n+1}=0$. Then
  $$u=p\varphi^j(0)\cdots \varphi^j(0)s=p \varphi^{j-1}(0)\varphi^{j-1}(1)\cdots \varphi^{j-1}(0)\varphi^{j-1}(1) s$$
  with $p$ (resp., $s$) a suffix (resp.,  prefix) of $\varphi^j(0)$.
If $|p|\ge 2^{j-1}$, then $p=p' \varphi^{j-1}(1)$ with $p'$ a
suffix of $\varphi^{j-1}(0)$ (and thus, a suffix of
$\varphi^{j}(1)$), otherwise set $p'=p\varphi^{j-1}(0)$.
Similarly, if $|s|\ge 2^{j-1}$, then $s=\varphi^{j-1}(0)s'$ with
$s'$ a prefix of $\varphi^{j-1}(1)$, otherwise
$s'=\varphi^{j-1}(1)s$. Notice that the corresponding
$\varphi^{j-1}$-factorization of $u$ is unique since the
$\varphi^{j-1}$-ancestor is not a power of a letter: if $n \neq 0$ then the claim is
clear. Otherwise $|u|=|ps| \geq 2^j - 1$; this implies that either
$|p| \geq 2^{j-1}$ or $|s| \geq 2^{j-1}$. Assuming the latter
(the other case being symmetric), we have that
$s = \varphi^{j-1}(0)s'$ with $s'$ a prefix of
$\varphi^{j-1}(1)$. If $s' \neq \varepsilon$, then the
$\varphi^{j-1}$-ancestor contains both letters. If
$s' = \varepsilon$, then $p \neq \varepsilon$, and then again
the $\varphi^{j-1}$-factorization contains both letters.
 
Now $u$ can also be written as
  \[
p'\varphi^{j-1}(1)\varphi^{j-1}(0)\cdots \varphi^{j-1}(1)\varphi^{j-1}(0) s'=p'\varphi^j(1)\cdots \varphi^j(1)s'.
  \]
There are no other $\varphi^j$-factorizations due to the
uniqueness of the $\varphi^{j-1}$ factorization of $u$. To
conclude the claim in this case, a straightforward case analysis
shows that $(p,s) \equiv_{j} (p',s')$ with $\big||p|-|p'|\big|=\big||s|-|s'|\big| = 2^{j-1}$:
  
  If $|p|\ge 2^{j-1}$ and if $|s|\ge 2^{j-1}$, then $(p,s)=(p' \varphi^{j-1}(1),\varphi^{j-1}(0)s')$.
  
  If $|p|\ge 2^{j-1}$ and if $|s|< 2^{j-1}$, then $(p,\varphi^{j-1}(1) s)=(p' \varphi^{j-1}(1),s')$.

   If $|p|< 2^{j-1}$ and if $|s|\ge 2^{j-1}$, then $(p\varphi^{j-1}(0),s)=(p',\varphi^{j-1}(0)s')$. 

      If $|p|< 2^{j-1}$ and if $|s|< 2^{j-1}$, then $(p\varphi^{j-1}(0),\varphi^{j-1}(1) s)=(p',s')$. 
\end{proof}



We have the following theorem, the proof of which is essentially the proof
of \cite[Thm.~48]{LejeuneLeroyRigo2020}. Indeed, the lemmas in
\cite{LejeuneLeroyRigo2020} leading to its proof do not require that the factors $u$ and $v$ are from the
Thue--Morse word, only that they have $\varphi^j$-factorizations.
We note that \cite[Thm.~48]{LejeuneLeroyRigo2020} is stated for
$j\geq 3$. However, the statement holds also for $j=1$
(trivially) and for $j=2$ as it is essentially a restatement
of \cite[Thm.~34]{LejeuneLeroyRigo2020} obtained by closely
inspecting its proof.

\begin{theorem}\label{thm:k-image-prefix-suffix-relation}
  Let $\infw{y}$ be an aperiodic binary word. 
Let $k\geq j\geq 1$. Let $u$ and $v$ be equal-length factors of $\infw{x}=\varphi^k(\infw{y})$ 
with $\varphi^j$-factorizations $u = p_1\varphi^j(z) s_1$ and $v = p_2 \varphi^j(z') s_2$. 
Then $u \sim_j v$ if and only if $(p_1,s_1)\equiv_j (p_2,s_2)$.
\end{theorem}

We then turn to the $(k+1)$-binomial equivalence in $\infw{x}$.
A straightforward consequence of \eqref{eq:morphic_image_coeff}
together with the identities
$\sum_{x \in A^{\ell}}\binom{u}{x} = \binom{|u|}{\ell}$,
$\ell \geq 1$, is the following observation.
\begin{lemma}\label{lem:coefficients_of_images}
Let $u \in \{0,1\}^*$. Then
\begin{equation*}
\binom{\varphi(u)}{0} = |u|;\quad
\binom{\varphi(u)}{01} = |u|_0 + \binom{|u|}{2};\quad
\binom{\varphi(u)}{011}  =\binom{u}{01} + \binom{|u|_0}{2} + \binom{|u|}{3}.
\end{equation*}
\end{lemma}
\begin{proof}
For example, $\binom{\varphi(a)}{011} = 0 = \binom{\varphi(a)}{11}$ for both $a \in \{0,1\}$. Similarly $\binom{\varphi(a)}{b} = 1$ for letters $a,b\in \{0,1\}$. Therefore
\begin{align*}
\binom{\varphi(u)}{011} &=
	\sum_{x_1, x_2 \in A}\binom{u}{x_1x_2}\sum_{\substack{011 = e_1e_2\\ e_i \in A^+}}
		\binom{\varphi(x_1)}{e_1}\binom{\varphi(x_2)}{e_2}
	+\sum_{|x| = 3}\binom{u}{x}\\
	&= \binom{u}{00} + \binom{u}{01} + \binom{|u|}{3} .
\end{align*}
and the claim follows.
\end{proof}

The next technical lemma has an important role in studying the
$(k+1)$-binomial equivalence.

\begin{lemma}\label{lem:Michel's}
  Let $u,v$ be two binary words of equal length. For $k\ge 1$, we have
  $$\binom{\varphi^k(u)}{01^k}-\binom{\varphi^k(v)}{01^k}=2^{(k-1)(k-2)/2}(|u|_0-|v|_0).$$
  In particular, $u\not\sim_1 v$ implies $\varphi^k(u)\not\sim_{k+1}\varphi^k(v)$. Moreover, if $u\sim_1 v$,  for $k\ge 1$, we have
  $$\binom{\varphi^k(u)}{01^{k+1}}-\binom{\varphi^k(v)}{01^{k+1}}=2^{(k-1)(k-2)/2}\left(\binom{u}{01}-\binom{v}{01}\right).$$
  In particular, $u\not\sim_2 v$ implies $\varphi^k(u)\not\sim_{k+2}\varphi^k(v)$.
\end{lemma}

\begin{proof}
  The case $k=1$ is deduced from \cref{lem:coefficients_of_images}. Then assume $k\ge 2$. We encourage the reader to refer to \cite{LejeuneLeroyRigo2020} for details that would be too long to reproduce here.  From \cite[Rem.~23]{LejeuneLeroyRigo2020}, we have the following expression 
  $$\binom{\varphi^k(u)}{01^k}-\binom{\varphi^k(v)}{01^k}=
  \sum_{x\in f^k(01^k)}m_{f^k(01^k)}(x)\left[\binom{u}{x}-\binom{v}{x}\right],$$
  where the map $f$ is defined to take into account the multiple ways factors $01$ or $10$ may occur in a word: $f(u)$ is a multiset of words of length shorter than $u$; see \cite[Def.~15 and 17]{LejeuneLeroyRigo2020}. We let the coefficient $m_{f^k(01^k)}(x)$ denote the multiplicity of $x$ as an element of the multiset $f^k(01^k)$. It can be shown that the multiset $f^k(01^k)$ only contains the elements $0$ and $1$. 
  Therefore we obtain
    $$\binom{\varphi^k(u)}{01^k}-\binom{\varphi^k(v)}{01^k}=m_{f^k(01^k)}(0)\, \left(|u|_0-|v|_0\right)+m_{f^k(01^k)}(1)\, \left(|u|_1-|v|_1\right).$$
 To conclude with the proof, we use two facts. The first is that $|u|_1-|v|_1=-(|u|_0-|v|_0)$ since $u,v$ have equal length.
 The second is that
  $$m_{f^k(01^k)}(0)-m_{f^k(01^k)}(1)=m_{f^{k-1}(01^k)}(01)-m_{f^{k-1}(01^k)}(10)=2^{(k-1)(k-2)/2},$$
  which follows from \cite[Prop.~28]{LejeuneLeroyRigo2020}. For the second part, the same reasoning may be applied to obtain
 $$\binom{\varphi^k(u)}{01^{k+1}}-\binom{\varphi^k(v)}{01^{k+1}}=
 \sum_{x\in f^k(01^{k+1})}m_{f^k(01^{k+1})}(x)\left[\binom{u}{x}-\binom{v}{x}\right].$$
 The multiset $f^k(01^{k+1})$ only contains $0,1,00,01,10,11$. But since it is assumed that $u\sim_1 v$, the only (potentially) non-zero terms
 in the sum correspond to $x\in\{01,10\}$. Then the observation
 $\binom{u}{01}-\binom{v}{01}=\binom{v}{10}-\binom{u}{10}$ following from \cref{lem:sum-constantPvect} suffices to conclude.
\end{proof}

Next we consider the structure of factors of the image of an arbitrary
binary word $\infw{y}$.

\begin{definition}
For $n\geq 1$ we let $\mathcal{S}(n) = \Fac_{n}(\infw{y})$. 
Further, for all $a,b\in \{\varepsilon,0,1\}$
such that $ab\neq \varepsilon$, we define
$\mathcal{S}_{a,b}(n) = \Fac_{n+|ab|}(\infw{y}) \cap a A^* b$.
We call these sets \emph{factorization classes of order $n$}.

Consider now a factor $u$ of
$\varphi(\infw{y})$.
We associate with $u$ some factorization classes as follows.
Let $a\varphi(u')b$ be the $\varphi$-factorization of $u$ with $\varphi$-ancestor $au'b \in \Fac(\infw{y})$.
If $ab = \varepsilon$, we associate the factorization class
$\mathcal{S}(|u'|)$. For $ab \neq \varepsilon$, we have that $u$ is a
factor of $\varphi(\overline{a} u' b)$. In this case we associate the
factorization class $\mathcal{S}_{\overline{a},b}(|u'|)$.
If $u$ is associated with a factorization class $\mathcal{T}$,
we write $u \assoc \mathcal{T}$, otherwise we write $u \not\models \mathcal{T}$.
\end{definition}

Observe that $u \assoc \mathcal{S}(n)$ implies that $|u| = 2n$.
Also, for $ab \neq \varepsilon$,
$u \assoc \mathcal{S}_{a,b}(n)$ implies that
$|u| = 2n+|ab|$.
Notice also that a factor $u$ of $\varphi(\infw{y})$ can be
associated with several factorization classes: take, e.g.,
$(10)^{\ell}1 = 1(01)^{\ell}$ which is associated with both
$\mathcal{S}_{\varepsilon,1}(\ell)$ and
$\mathcal{S}_{0,\varepsilon}(\ell)$, or $(01)^{\ell+1} = 0(10)^{\ell}1$ which is
associated with both $\mathcal{S}(\ell + 1)$ and $\mathcal{S}_{1,1}(\ell)$.

\begin{lemma}\label{lem:2-bin_same_fac-class}
For two $2$-binomially equivalent factors
$u,v \in \Fac(\varphi(\infw{y}))$, if $u \assoc \mathcal{T}$
for some factorization class $\mathcal{T}$, then $v \assoc \mathcal{T}$. Furthermore, a factor $u$ of $\infw{y}$ is associated with distinct factorization classes if
and only if $u \in L = (01)^* \cup (10)^* \cup 1(01)^* \cup 0(10)^*$.
\end{lemma}
\begin{proof}
\textbf{Even-length factors.}
Let $u \sim_2 v$ with $|u| = 2n$. If
$u \assoc \mathcal{S}_{\overline{a},a}(n-1)$ with $a \in \{0,1\}$,
then $u$ is of the form $a\varphi(x)a$ with $|x| = n-1$, whence
$|u|_a = n+1$. Factors
$v' \not\assoc \mathcal{S}_{\overline{a},a}(n-1)$ of length $2n$
have $|v'|_a \leq n$ by inspection. Hence also
$v \assoc \mathcal{S}_{\overline{a},a}(n-1)$.
The above arguments
also show that $u$ is associated with exactly one factorization
class. For the latter claim, we note that $u$ has even length and begins and ends with the same letter, so it cannot appear in the language $L$.

Assume then that $u \not\assoc \mathcal{S}_{\overline{a},a}(n-1)$,
$a\in \{0,1\}$. Then
$v \not\assoc \mathcal{S}_{\overline{a},a}(n-1)$, $a\in \{0,1\}$
by the previous observation. Notice that we may assume
$n\geq 2$ as otherwise we have $|u| = 2$ and the claim is
trivial ($2$-binomial equivalence is equality in this case). We compare
the values of $\binom{y}{01}$ for $y$ associated with
$\mathcal{S}_{1,1}(n-1)$, $\mathcal{S}_{0,0}(n-1)$, and
$\mathcal{S}(n)$, respectively.

\begin{enumerate}[wide=15pt,widest=99,label=\textbf{Case \arabic*:},itemsep=3pt]

\item 
$y\assoc \mathcal{S}_{1,1}(n-1)$.
We have $\binom{y}{01} \geq \binom{n}{2}+n$, and equality holds for $y = (01)^n$.
Indeed, say $y = 0\varphi(x)1$ for some $x\in \{0,1\}^{n-1}$. Then
we have by \cref{lem:coefficients_of_images}
\begin{equation*}
\binom{y}{01} 
= \binom{\varphi(x)}{01} + |\varphi(x)|_0 + |\varphi(x)1|_{1} = |x|_0 + \binom{|x|}{2} + 2|x| + 1 = |x|_0 + \binom{n}{2} + n,
\end{equation*}
since $|x| = n-1$. Equality now holds when $|x|_0 = 0$, i.e., $x = 1^{n-1}$.

\item
$y \assoc \mathcal{S}_{0,0}(n-1)$.
We have $\binom{y}{01} \leq \binom{n}{2}$, and equality holds when $y = (10)^n$. Indeed, say $y = 1\varphi(x)0$ for some $x\in \{0,1\}^{n-1}$. Then
\begin{equation*}
\binom{y}{01} = \binom{\varphi(x)}{01} = |x|_0 + \binom{|x|}{2} = |x|_0 + \binom{n}{2} - (n-1).
\end{equation*}
Since $|x| = n-1$, we have $\binom{y}{01} \leq \binom{n}{2}$. Equality holds when $x = 0^{n-1}$.

\item
$y \assoc \mathcal{S}(n)$. 
We have $\binom{n}{2} \leq \binom{y}{01} \leq \binom{n}{2} + n$.
The former equality is attained with $y = (10)^n$ and the latter with $y = (01)^n$. Indeed, say
$y = \varphi(x')$ for some $x'\in \{0,1\}^n$. We have $\binom{y}{01} = \binom{n}{2} + |x'|_0$ from
\cref{lem:coefficients_of_images}. Therefore, $\binom{n}{2} \leq \binom{y}{01} \leq \binom{n}{2} + n$.
The former equality is attained with $x' = 1^n$ and the latter with $x'= 0^n$.
\end{enumerate}
We conclude that $u$ and $v$ are associated with a common factorization class. In fact, the latter claim is also implied from the above:
a word can be associated with two (and only two) factorization classes if and only if it appears in $L$.
This concludes the proof in the case of even length factors.

\textbf{Odd-length factors}.
Assume without loss of generality that
$u \assoc \mathcal{S}_{a,\varepsilon}(n)$ with
$u = a\varphi(u')$ of length $2n + 1$.
Recalling that $|\varphi(u')|_0 = |u'| = n$, if $u \sim_2 v$
with $u$ and $v$ associated with distinct factorization classes,
then necessarily $v \in \mathcal{S}_{\varepsilon,a}$, say
$v = \varphi(v')a$.
We show that this is impossible, unless $u = v \in L$.

Indeed, assuming that we have $2$-binomial equivalence, we have
\begin{equation}\label{eq:odd_length_01_ximage}
\binom{a\varphi(u')}{01} = \binom{\varphi(u')}{01} + \delta_0(a) \binom{\varphi(u')}{1}
= |u'|_0 + \binom{n}{2} + \delta_0(a) n
\end{equation}
which is equal to
\begin{equation}\label{eq:odd_length_01_image_x}
\binom{\varphi(v')a}{01} =  \binom{\varphi(v')}{01} + \delta_1(a)\binom{\varphi(v')}{0}
= |v'|_0 + \binom{n}{2} + \delta_1(a) n
\end{equation}
where $\delta_a(b) = 1$ if $a= b$, otherwise $\delta_a(b) = 0$.
Rearranging, we get $|u'|_0 - |v'|_0 = (\delta_1(a) - \delta_0(a))n \in \{ \pm n \}$. This implies, without loss of generality, that $u' = 0^n$, $v' = 1^n$, and $a = 1$.
But then $u = 1(01)^n = (10)^n1 = v \in L$, as claimed.
\end{proof}

The next result characterizes $(k+1)$-binomial equivalence in $\infw{x} = \varphi^k(\infw{y})$ when $\infw{y}$
is an arbitrary binary word.
\begin{proposition}\label{pro:k-image_k+1-prefix-suffix}
Let $u$ and $v$ be factors of length at least $2^k-1$ of $\infw{x}=\varphi^k(\infw{y})$ with the $\varphi^k$-factorizations
$u = p_1 \varphi^k(z) s_1$ and $v = p_2 \varphi^k(z') s_2$.
Then $u \sim_{k+1} v$ and $u\neq v$ if and only if $z\sim_1 z'$, $z'\neq z$, and $(p_1,s_1) = (p_2,s_2)$.
\end{proposition}
Notice that the proposition claims that those factors of $\infw{x}$ having at least two $\varphi^k$-factorizations are $(k+1)$-binomially equivalent only to themselves (in
$\Fac(\infw{x})$).
\begin{proof}
The ``if''-part of the statement follows by a repeated application
of \cref{prop:Parikh-collinear-k-characterization} on the Thue--Morse morphism together with the fact that the morphism is injective.

Let us assume that $u\sim_{k+1} v$ for some distinct factors. It follows that $u \sim_{k} v$, which implies
that $(p_1,s_1) \equiv_k (p_2,s_2)$ by \cref{thm:k-image-prefix-suffix-relation}. Next we show that $(p_1,s_1) = (p_2,s_2)$ and
$z \sim_1 z'$. We have the following case distinction from
\cref{def:type}:

\eqref{it:same_sum}(a): We have that $(p_1,s_1) = (p_2,s_2)$.
By deleting the common prefix $p_1$ and suffix $s_1$,
we are left with the equivalent statement
$\varphi^{k}(z) \sim_{k+1} \varphi^{k}(z')$. If $z \not \sim_1 z'$, then we have a contradiction with \cref{lem:Michel's}.
The desired result follows in this case.

In the remaining cases, we assume towards a contradiction that $(p_1,s_1) \neq (p_2,s_2)$.

\eqref{it:same_sum}(b): Suppose that
$(p_1,s_2) = (p_2\varphi^{k-1}(a),\varphi^{k-1}(a)s_1)$.
Deleting the common prefixes $p_2$ and suffixes $s_1$, we are left with
$\varphi^{k-1}(a\varphi(z)) \sim_{k+1} \varphi^{k-1}(\varphi(z') a)$.
Now $a\varphi(z) \sim_1 \varphi(z')a$, but $a\varphi(z) \not\sim_2 \varphi(z')a$ by \cref{lem:2-bin_same_fac-class} (otherwise
$a\varphi(z) = \varphi(z')a$ and thus $u = v$ contrary to the assumption). \cref{lem:Michel's}
then implies that $\varphi^{k-1}(a\varphi(z)) \not\sim_{k+1} \varphi^{k-1}(\varphi(z')a)$, which is a contradiction.

\eqref{it:same_sum}(c): Suppose that 
$(p_2, \varphi^{k-1}(a) s_2) = (p_1 \varphi^{k-1}(a), s_1)$. This is symmetric to the previous case.

\eqref{it:same_sum}(d): Suppose that
$(p_1,s_1) = (s_2,p_2) = (\varphi^{k-1}(a),\varphi^{k-1}(\overline{a}))$. We thus have directly
\[
\varphi^{k-1}(a\varphi(z)\overline{a}) \sim_{k+1} \varphi^{k-1}(\overline{a}\varphi(z)a).
\]
The claim follows by an argument similar to that of in Case \eqref{it:same_sum}(b).

\eqref{it:not_same_sum}(a): Suppose that
$(p_1,s_1) = (p_2\varphi^{k-1}(a),\varphi^{k-1}(\bar{a})s_2)$. After removing common prefixes and suffixes,
we are left with
$ \varphi^{k-1}(a\varphi(z)\overline{a})  \sim_{k+1} \varphi^{k-1}(\varphi(z'))$. We have that
$a\varphi(z)\overline{a} \sim_1 \varphi(z')$, but by \cref{lem:2-bin_same_fac-class}
$a\varphi(z)\overline{a} \not\sim_2 \varphi(z')$ (otherwise $z = \overline{a}^{\ell}$ and $z' = a^{\ell + 1}$, implying that $u=v$, a contradiction). This is again a contradiction by \cref{lem:Michel's}.

\eqref{it:not_same_sum}(b): Suppose that
$(p_2,s_2) = (p_1\varphi^{j-1}(a),\varphi^{j-1}(\bar{a})s_1)$.
This is symmetric to the previous case.
\end{proof}

Notice that \cref{thm:k-image-j-complexities} and \cref{pro:k-image_k+1-prefix-suffix} have the following corollary:
\begin{corollary}\label{cor:k-image_k+1-prec}
Let $\infw{x} = \varphi^k(\infw{y})$, where $\infw{y}$ is an arbitrary
aperiodic binary word. We have
\[ \bc{\infw{x}}{1} \prec \bc{\infw{x}}{2} \prec \cdots \prec \bc{\infw{x}}{k} \prec \bc{\infw{x}}{k+1}. \]
\end{corollary}
\begin{proof}
Recall that $\infw{y}$ contains arbitrarily long factors of the
form $\overline{a}za$, $a\in \{0,1\}$ \cref{lem:boundaries}. Therefore $\infw{x}$ contains
the $k$-binomially equivalent (by \cref{lem:transfer}) factors
$\varphi^{k-1}(a)\varphi^k(z)$ and
$\varphi^k(z)\varphi^{k-1}(a)$. However, by \cref{pro:k-image_k+1-prefix-suffix} these factors are either not
$(k+1)$-binomially equivalent, or
$\varphi^{k-1}(a)\varphi^k(z) = \varphi^k(z)\varphi^{k-1}(a)$.
The latter happens when $\varphi^k(z) = \varphi^{k-1}(a)^{\ell}$ for some $\ell \geq 0$, and thus only when $\ell = 0$ and $z = \varepsilon$.
(Indeed, it is not hard to prove that if $w$ is primitive so is $\varphi(w)$.)
This observation suffices for showing
$\bc{\infw{x}}{k} \prec \bc{\infw{x}}{k+1}$. The rest of the claim follows by \cref{thm:k-image-j-complexities}.
\end{proof}

\section{Binomial Properties of the Thue--Morse Morphism, Part II}
\label{sec:propsTMII}
In this section we consider a complementary result to \cref{thm:k-image-j-complexities}, which partially extends the following
theorem of Richomme, Saari, and Zamboni
\cite{RichommeSZ2011Abelian}.
\begin{theorem}[{\cite[Thm~3.3]{RichommeSZ2011Abelian}}]
\label{thm:richommeTM}
Let $\infw{x}$ be an aperiodic word.
Then $\bc{\infw{x}}{1} = \bc{\infw{t}}{1}$ if and only if there
exists a binary word $\infw{y}$ and $a \in \{\varepsilon,0,1\}$
such that $\infw{x} = a\varphi(\infw{y})$.
\end{theorem}
Notice that \cref{thm:k-image-j-complexities} is a generalization
of the ``if''-direction. We give a partial generalization in the other direction in \cref{sec:converse}, namely:
\begin{theorem}\label{thm:k-complexities-k-image}
Let $\infw{x}$ be a recurrent binary word having property
$\tmprop(k)$ for some $k\geq 1$. Then there exists a binary
word $\infw{y}$ such that $\infw{x} = u\varphi^k(\infw{y})$,
where $u$ is a proper suffix of $\varphi^k(0)$ or
$\varphi^k(1)$.
\end{theorem}


To prove the theorem, we first derive a formula for counting
$(k+1)$-binomial equivalence classes of words that are of the form
$\varphi^{k}(\infw{y})$ for $\infw{y}$ aperiodic in \cref{sec:formula}.

\subsection{A formula for counting \texorpdfstring{$(k+1)$}{(k+1)}-binomial complexities}
\label{sec:formula}

For a binary word $\infw{y}$ we define
\begin{align*}
X_{\infw{y}}(n) &:= \{(a,\Psi(u),b) \colon a,b \in \{0,1\}, aub \in \Fac_{n+1}(\infw{y})\},\\
Y_{\infw{y},L}(n) &:= \{(a,\Psi(u)) \colon a \in \{0,1\}, au \in \Fac_{n+1}(\infw{y})\}, \\
Y_{\infw{y},R}(n) &:= \{(\Psi(u),a) \colon a \in \{0,1\}, ua \in \Fac_{n+1}(\infw{y})\},\\
Y_{\infw{y}}(n) &:= Y_{\infw{y},L} \cup Y_{\infw{y},R}.
\end{align*}

\begin{observation}
Let $G_n = (V,E)$ be the abelian Rauzy graph of $\infw{y}$ (of order $n$).
\begin{itemize}
\item $E$ is in one-to-one correspondence with
$X_{\infw{y}}(n)$, namely, $\vec{x}\xrightarrow{(a,b)}\vec{y}$ is identified with $(a,\vec{x}-\Psi(a),b)$. In particular, $\# E=\# X_{\infw{y}}(n)$.

\item The set $Y_{\infw{y},R}(n)$ is in one-to-one correspondence with
$E/{\equiv_R}$, where $\equiv_R$ is the equivalence relation defined by
the (surjective) mapping $E \to Y_{\infw{y},R}(n)$ (meaning, the equivalence classes are the full preimages of elements of $Y_{\infw{y},R}(n)$),
\[
\Big( \vec{x} \xrightarrow{(a,b)} \vec{y} \Big)
\mapsto (\vec{x},b).
\]
In particular, $\# Y_{\infw{y},R}(n)\le \# E$. 
Similarly $Y_{\infw{y},L}(n)$ is in one-to-one correspondence with
$E/{\equiv_L}$, where $\equiv_L$ is the equivalence relation defined by
the (surjective) mapping
\[
\Big( \vec{x} \xrightarrow{(a,b)} \vec{y} \Big)
\mapsto (a,\vec{y}).
\]
In particular, $\# Y_{\infw{y},L}(n)\le \# E$.

\item Each equivalence class in $E/{\equiv_L}$ contains at most two
elements: two edges are equivalent if their target vertices and the first
components of the labels are equal. Hence the equivalence relation
can only identify a non-loop edge with a loop.

\item Note that any loop at a vertex $v$ with label $(0,0)$ can only be equivalent
under either $\equiv_L$ or $\equiv_R$ to an edge between $v$ and a lighter
vertex. Similarly a loop with label $(1,1)$ can only be equivalent to an edge
between $v$ and a heavier vertex. In particular, a loop with $(0,0)$ (resp., $(1,1)$)
on the lightest (resp., heaviest) vertex is not equivalent to any other edge
under either $\equiv_L$ or $\equiv_R$.

\end{itemize}
\end{observation}

\begin{example}
Recall the abelian Rauzy graphs of the Thue--Morse word from \cref{ex:abelian-Rauzy-TM}.
The edges correspond exactly to $X_{\infw{t}}(n)$.
The equivalence classes of $E/{\equiv_R}$ (resp., $E/{\equiv_L}$)
corresponding to $Y_{\infw{t},R}(n)$ (resp., $Y_{\infw{t},L}(n)$) containing at least two elements are listed below:
\begin{itemize}[align=left,leftmargin=*]
\item[$n=2m$:] $Y_{\infw{t},R}:$ $\left\{ m \xrightarrow{(0,0)} m, m \xrightarrow{(1,0)} m-1 \right\}$,\quad 
$\left\{ m \xrightarrow{(1,1)} m, m \xrightarrow{(0,1)} m+1 \right\}$;

$Y_{\infw{t},L}:$ $\left\{ m \xrightarrow{(0,0)} m, m-1 \xrightarrow{(0,1)} m \right\}$,\quad 
$\left\{ m \xrightarrow{(1,1)} m, m+1 \xrightarrow{(1,0)} m \right\}$.

\item[$n=2m-1$:] $Y_{\infw{t},R}:$ $\left\{ m \xrightarrow{(1,1)} m, m \xrightarrow{(0,1)} m+1 \right\}$, \quad
$\left\{ m+1 \xrightarrow{(0,0)} m+1,m+1 \xrightarrow{(1,0)} m \right\}$;

$Y_{\infw{t},L}:$ $\left\{ m \xrightarrow{(1,1)} m, m+1 \xrightarrow{(1,0)} m\right\}$,\quad
$\left\{ m+1 \xrightarrow{(0,0)} m+1, m \xrightarrow{(0,1)} m+1 \right\}$.

\end{itemize}
\end{example}

We may now establish a formula for counting the $(k+1)$-binomial
complexity of the $k$th image of a word $\infw{y}$ under the
Thue--Morse morphism. This will turn
out to be key in proving a converse to
\cref{thm:k-image-j-complexities}.

\begin{proposition}
\label{prop:formula}
Let $\infw{y}$ be an infinite binary word and let $\infw{x} = \varphi^k(\infw{y})$ with $k\geq 1$.
Let
$m = \max\{n \in \N \colon 0^{n} \text{ and } 1^{n} \in \Fac(\infw{y})\}$
and $m' = \max \{n \in \N \colon 0^{n} \text{ or } 1^{n} \in \Fac(\infw{y})\}$,
where we allow $m$ and $m'$ to equal $\infty$. We have
$\bc{\infw{x}}{k+1}(r) = \fc{\infw{t}}(r)$ for all $0\leq r < 2^k$.
Setting $Z(n,0) := (2^k-1)\#X_{\infw{y}}(n) + \bc{\infw{y}}{1}(n)$, for all $n \geq 1$ we have
\begin{equation}\label{eq:k+1-formula}
\bc{\infw{x}}{k+1}(2^k n) = Z(n,0) -
\begin{cases} 2^k, & \text{if } n < m; \\
1,  & \text{if } n=m < m'; \\
0, & \text{otherwise}.
\end{cases}
\end{equation}

For all $n \geq 1$ and $0 < r < 2^k$, setting $Z(n,r) := (r-1)\# X_{\infw{y}}(n+1) +  (2^k-r-1)\#X_{\infw{y}}(n) + \#Y_{\infw{y}}(n)$, we have
\begin{equation}\label{eq:k+1-formula2}
\bc{\infw{x}}{k+1}(2^k n + r) = Z(n,r) -
\begin{cases}
2^k, & \text{if } n + 1 < m;\\
(2^k - r + 1), & \text{if } n+1 = m < m';\\
(2^k - 2(r-1)), & \text{if } n+1 = m =m' \text{ and } r\leq 2^{k-1};\\
0, & \text{otherwise}.
\end{cases}
\end{equation}
\end{proposition}
\begin{proof}
\cref{lem:k-image-short-factors} implies the formula for lengths less than $2^k$.
The proof strategy to establish formulas \cref{eq:k+1-formula,eq:k+1-formula2} is as follows.
For $n\ge 1$ and $0\le r < 2^k$,  we first obtain an upper bound $Z(n,r)$ on $\bc{\infw{x}}{k+1}(2^k n + r)$ by counting
the different $\varphi^k$-factorizations up to the equivalence implied by \cref{pro:k-image_k+1-prefix-suffix}.
Then we establish the exact formula by subtracting the number of $(k+1)$-binomial classes that admit several
$\varphi^k$-factorizations as counted above. 
To do so,  we use the following argument.
By \cref{pro:k-image_k+1-prefix-suffix} and the observation made right after its statement,  those factors of $\infw{x}$
that admit several $\varphi^k$-factorizations are $(k+1)$-binomially equivalent only to themselves. In fact, such factors
are well-understood by \cref{lem:unique_image}; they only admit two distinct $\varphi^k$-factorizations.
Hence,  counting the number of factors that have two $\varphi^k$-factorizations and subtracting that number from
the term $Z(n,r)$ gives the number of $(k+1)$-binomial equivalence classes.

We first prove formula~\cref{eq:k+1-formula} by inspecting factors of length $2^k n$ for some $n \geq 1$.
They are of the following two forms: either $\varphi^k(u)$,  with $|u|=n$,  or $p \varphi^k(v) s$,  with $|v|=n-1$, $p$ and $s$ non-empty.
Each abelian equivalence class in $\Fac_n(\infw{y})/{\sim_1}$ gives a $(k+1)$-equivalence class of factors of the first form by \cref{prop:Parikh-collinear-k-characterization} (recall that $\varphi$ is Parikh-collinear).
Hence the term $\bc{\infw{y}}{1}(n)$ in $Z(n,0)$.
For the factors of the second form, we notice the following. Such a factor has the $\varphi^k$-ancestor $avb$, with
$(a,\Psi(v),b) \in X_{\infw{y}}(n)$. 
On the other hand, any $(a,\Psi(v),b) \in X_{\infw{y}}(n)$ gives rise to
$(2^k-1)$ $(k+1)$-binomial equivalence classes, namely,  those represented by the words
\[
\suff_i(\varphi^k(a))\ \varphi^k(v)\ \pref_{2^k-i}(\varphi^k(b)),
\ 1 \leq i  < 2^k,
\]
where,  for a word $w$ and $i\in\{1,\ldots,|w|\}$,  we let $\pref_i(w)$ (resp.,  $\suff_i(w)$) denote the length-$i$ prefix (resp.,  suffix) of $w$.
Hence the term $(2^k-1)\#X_{\infw{y}}(n)$ in the formula.
Therefore we have established the upper bound $\bc{\infw{x}}{k+1}(2^k n) \le Z(n,0)$, with $Z(n,0)=(2^k-1)\#X_{\infw{y}}(n) + \bc{\infw{y}}{1}(n)$.

As explained at the beginning of the proof,  we now examine factors admitting several $\varphi^k$-factorizations and subtract their number from $Z(n,0)$ to establish formula~\cref{eq:k+1-formula}.
Let $x$ be such a factor. 
Then it has, by \cref{lem:unique_image}, exactly two $\varphi^k$-factorizations, and we may write
\begin{equation}\label{eq:2facs}
p\varphi^k(u)s = x = p'\varphi^k(u')s'.
\end{equation}
Here we note that $|ps|=2^k$ if and only if $p$ or $s$ non-empty.
Moreover,
the corresponding $\varphi^k$-ancestors are powers of letters, and given in the table below (where in the third row $p'$ is defined by $p = p'\varphi^{k-1}(\overline{a})$, and in the fourth row,  $s'$ is defined by $s = \varphi^{k-1}(a)s'$).
\[
\begin{array}{c|c|c|c|c}
\text{fact.} & \text{ancest.} & \text{conditions on } p,s & \text{2nd fact.} & \text{ancest.} \\
\hline
p\varphi^k(a^{n})s
&
a^{n}
&
|p| = 0 =|s|
&
\varphi^{k-1}(a) \varphi^k(\overline{a}^{n-1})\varphi^{k-1}(\overline{a})
& 
\overline{a}^{n+1} 
\\
p\varphi^k(a^{n-1})s
&
a^{n+1}
&
p = \varphi^{k-1}(\overline{a}), s = \varphi^{k-1}(a)
&
\varphi^{k}(\overline{a}^{n})
& 
\overline{a}^{n} 
\\
p\varphi^k(a^{n-1})s
&
a^{n+1}
&
|p| > 2^{k-1},0 < |s| < 2^{k-1}
&
p'\varphi^k(\overline{a}^{n-1})\varphi^{k-1}(\overline{a})s
& 
\overline{a}^{n+1}
\\
p\varphi^k(a^{n-1})s
&
a^{n+1}
&
0< |p| < 2^{k-1}, |s| > 2^{k-1}
 &
p\varphi^{k-1}(a)\varphi^k(\overline{a}^{n-1})s'
&
\overline{a}^{n+1}
\end{array}
\]
In particular,  for $x$ to have two $\varphi^k$-factorizations (and so for a class to have been counted twice), $a^{n}$ and $\overline{a}^n$ both must appear in $\infw{y}$, and at least one of $a^{n+1}$ and $\overline{a}^{n+1}$ has to also appear in the word.
We divide the proof into three cases. 

\begin{enumerate}[wide=15pt,widest=99,label=\textbf{Case \arabic*.},itemsep=3pt]

\item If $n > m$ or when $n=m=m'$,  then as concluded above,  there is no factor having several $\varphi^k$-factorizations,  and the formula holds.

\item Assume now that $n < m$, so both $a^{n+1}$ and $\overline{a}^{n+1}$ appear in
$\infw{y}$.
Reusing the table above,  we see that any equivalence class corresponding to a
factor having $\varphi^k$-ancestor $a^{n+1}$ (or $a^n$) has been counted twice (and corresponds to a word with $\varphi^k$-ancestor $\overline{a}^{n+1}$ or $\overline{a}^n$). 
There are $2^k$ of those, whence the formula for $n < m$.

\item Assume finally that $n = m$ and $m' > m$. 
Assume without loss of generality that $a^{m+1}$ appears in $\infw{y}$.
Therefore $\overline{a}^{m+1}$ does not appear in $\infw{y}$.
Thus, if $x$ has two $\varphi^k$-ancestors, one of them is $\overline{a}^m$.
There is only one such factor, and this proves the remaining case in the formula.
\end{enumerate}

We now turn to the proof of formula \cref{eq:k+1-formula2}, and consider
factors of the length $2^kn + r$ , with $n\ge 1$ and $0<r<2^k$. 
Let us first establish the upper bound $Z(n,r)$ on $\bc{\infw{x}}{k+1}(2^k n + r)$.
Each element of $Y_{\infw{y}}(n)$ gives rise to unique $(k+1)$-factorization;
for example, $(a,\Psi(u))$ gives the class represented by $\suff_{r}(\varphi^k(a)) \varphi^k(u)$.
Hence the term $\#Y_{\infw{y}}(n)$ in $Z(n,r)$.
Each element of $X_{\infw{y}}(n+1)$ gives rise to
$r-1$ many $(k+1)$-factorizations as follows: $(a,\Psi(u),b)$ gives
\[
\suff_{i}(\varphi^k(a))\ \varphi^k(u)\ \pref_{r-i}(\varphi^k(b)), \quad 1 \leq i  < r. 
\]
Similarly each element of $X_{\infw{y}}(n)$ gives $2^{k}-r-1$
elements, namely $(a,\Psi(u'),b)$ gives
\[
\suff_{i}(\varphi^k(a))\ \varphi^k(u')\ \pref_{2^k + r - i}(\varphi^k(b)) \quad r < i < 2^k.
\]
Hence $\bc{\infw{x}}{k+1}(2^k n + r) \le Z(n,r)$ with
$Z(n,r) = (r-1)\# X_{\infw{y}}(n+1) +  (2^k-r-1)\#X_{\infw{y}}(n) + \#Y_{\infw{y}}(n)$.

We again face the problem of over-counting. 
As previously,  we count the number of factors that have two $\varphi^k$-factorizations. 
Let again $x$ have two $\varphi^k$-factorizations as in \cref{eq:2facs}, where now $ps$ and $p's'$ are both non-empty. 
The $\varphi^k$-ancestor of each factorization is given in the table below.
\[
\begin{array}{c|c|l}
\text{fact.} & \text{ancestor} & \text{conds.~on }p,s\\
\hline
p\varphi^k(u)s 
& 
\begin{array}{c}
a^{n+2} \\
a^{n+1} \\
\end{array}
&
\begin{array}{l}
\text{if } |ps|=r \text{ and } p,s\neq \varepsilon \\
\text{if } |ps|=r \text{ with } p=\varepsilon \text{ or } s=\varepsilon,  \text{ or if } |ps|>r
\end{array}
\\
\hline
p'\varphi^k(u')s'
&  
\begin{array}{c}
\overline{a}^{n+2} \\
\overline{a}^{n+1} \\
\end{array}
&
\begin{array}{l}
\text{if } |p's'|=r \text{ and } p',s'\neq \varepsilon \\
\text{if }|p's'|=r \text{ with } p'=\varepsilon \text{ or } s'=\varepsilon,  \text{ or if } |p's'|>r
\end{array}
\end{array}
\]
We conclude that,  for $x$ to have two $\varphi^k$-factorizations (and so for a class to have been counted twice), we must have $n+1 \leq m$.
We divide the proof into three cases. 

\begin{enumerate}[wide=15pt,widest=99,label=\textbf{Case \arabic*.},itemsep=3pt]

\item Assume that $n+1 > m$.  As concluded above,  there is no factor having several $\varphi^k$-factorizations,  and the formula holds.

\item Assume that $n+1 < m$. In this case we have $1^{n+2}$ and $0^{n+2}$
appearing in $\infw{y}$.
We claim that any factor with a $\varphi^k$-ancestor $a^{n+2}$ or $a^{n+1}$ has also a $\varphi^k$-ancestor $\overline{a}^{n+2}$ or $\overline{a}^{n+1}$,  and we show there are $2^k$ such factors.
Hence the formula follows.

First, if $x = p\varphi^k(a^n)s$, with $|ps|=r$, we have $x = p\varphi^{k-1}(a)\varphi^k(\overline{a}^{n-1})\varphi^{k-1}(\overline{a})s$.  
The other $\varphi^k$-factorization of $x$ is given in the table below, where $p'$ and $s'$ are suitably chosen (for example, in the first case,  we have $p'=p\varphi^{k-1}(a)$ and $s'=\varphi^{k-1}(\overline{a})s$).
\[
\begin{array}{c|c|c}
\text{conditions on } p,s & \text{$\varphi^k$-fact.} & \text{$\varphi^k$-ancestor} \\
\hline
|p|,|s| < 2^{k-1} 
&
p'\varphi^{k}(\overline{a}^{n-1})s'
& 
\overline{a}^{n+1} 
\\
|p| \ge 2^{k-1} \text{ and } |s| < 2^{k-1} 
&
p'\varphi^k(\overline{a}^{n})\varphi^{k-1}(\overline{a})s
& 
\overline{a}^{n+2}
\\
|p| < 2^{k-1} \text{ and } |s| \ge 2^{k-1} 
 &
p\varphi^{k-1}(a)\varphi^k(\overline{a}^{n})s'
&
\overline{a}^{n+2}
\end{array}
\]
(Observe that $|p| \geq  2^{k-1}$ and $|s| \geq 2^{k-1}$ cannot simultaneously hold as $|ps| = r < 2^k$.)

Second,  if $x = p\varphi^k(a^{n-1})s$ with $|ps| = 2^k + r$, $r < |p| < 2^k$,  the other $\varphi^k$-factorization of $x$ is given in the table below,  where $p'$ and $s'$ are again suitably chosen.
\begin{equation}\label{eq:2facsan1}
\begin{array}{c|c|c}
\text{conditions on } p,s & \text{$\varphi^k$-fact.} & \text{$\varphi^k$-ancestor} \\
\hline
|p| \geq 2^{k-1}, |s| < 2^{k-1}
&
p'\varphi^{k}(\overline{a}^{n-1}) \varphi^{k-1}(\overline{a})s
& 
\overline{a}^{n+1} 
\\
|p|=2^{k-1}, |s|=2^{k-1}+r
&
\varphi^{k}(\overline{a}^{n})s'
& 
\overline{a}^{n+1} 
\\
|p|,|s| > 2^{k-1}
&
p'\varphi^{k}(\overline{a}^{n})s'
& 
\overline{a}^{n+2}
\\
|p|=2^{k-1}+r, |s|=2^{k-1}
&
p'\varphi^{k}(\overline{a}^{n})
& 
\overline{a}^{n+1} 
\\
 |p| < 2^{k-1}, |s| \geq 2^{k-1}
 &
p\varphi^{k-1}(a)\varphi^{k}(\overline{a}^{n-1}) s'
&
\overline{a}^{n+1}
\end{array}
\end{equation}
This concludes the proof for this part, as we have exhibited $2^k$ distinct factors, and there are no other possibilities. (Indeed, there
are $r-1$ factors having $a^{n+2}$ as a $\varphi^k$-ancestor,
and $2^k-r+1$ factors having $a^{n+1}$ as such.)

\item Assume finally that $n+1 = m$.  We divide the proof into two subcases.

\begin{enumerate}[wide=15pt,widest=99,label*=\textbf{\arabic*.},itemsep=1pt]
\item Assume that $m < m'$. 
Let us assume that $a^{n+1}$ appears in $\infw{y}$ but $a^{n+2}$ does not. Then $\overline{a}^{n+2}$ does under the assumption. Notice that in the previous case there were exactly $r-1$ factors having $a^{n+2}$ as a $\varphi^k$-ancestor. 
Under our current assumption, these factors do not have this ancestor, but have instead the ancestor $\overline{a}^{n+2}$.
They thus have only one $\varphi^k$-factorization. 
However,  as before,  the $2^{k-1}-r+1$ factors with $\varphi^k$-ancestor $a^{n+1}$ have a second $\varphi^k$-factorization. We conclude that the formula holds also in this case.

\item Finally assume that $m' = m$.
Then we have that neither $a^{n+2}$ nor $\overline{a}^{n+2}$ appears in $\infw{y}$, while both $a^{n+1}$ and $\overline{a}^{n+1}$ do. 
We thus need to count those factors that have both $a^{n+1}$ and $\overline{a}^{n+1}$ as $\varphi^k$-ancestors. 
Looking at the previous table,  only the center row gives $\overline{a}^{n+2}$ as a $\varphi^k$-ancestor. 
Such factors appear when $|ps| = 2^k + r$ with $2^{k-1} < |p| < 2^{k-1} + r$, i.e., there are $r-1$ of them whenever $r \leq 2^{k-1}$ (recall that $|p|,|s|<2^k$).
Symmetric arguments apply to factors having $\varphi^k$-ancestors
$\overline{a}^{n+1}$ (i.e.,  exchanging the role of $a$ and $\overline{a}$). 
We conclude that the number of factors having two $\varphi^k$-factorizations is $2^k-2(r-1)$ when $r \leq 2^{k-1}$, as is claimed in the formula. 

We are left with the case that $r > 2^{k-1}$.  
Here we show that no factor has two $\varphi^k$-factorizations with respect to $\infw{y}$.
Since $r > 2^{k-1}$,  we have $|ps| = 2^k + r > 2^k + 2^{k-1}$ with $|p|,|s| < 2^{k}$.
It follows that for such $\varphi^k$-factorizations we must have $|p|,|s| > 2^{k-1}$,  which only leaves the center row of the previous table. 
But,  we already discarded these factors,  so the proof is completed.\qedhere
\end{enumerate}
\end{enumerate}

\subsection{A converse to \texorpdfstring{\cref{thm:k-image-j-complexities}}{}}
\label{sec:converse}

As announced at the beginning of the section, we now obtain a
partial converse statement to \cref{thm:k-image-j-complexities}.
Before giving the proof, which is quite long and technical, we give a brief sketch of it.
The proof is by induction on $k$. The induction hypothesis allows to conclude that
$\infw{x}$ in the statement is essentially the $k$th image of a recurrent word $\infw{z}$.
We then show that $\infw{z}$ has property $\tmprop(1)$ using several times the
formulas established in \cref{prop:formula}. The word $\infw{z}$ having property
$\tmprop(1)$ allows to show that $\infw{x}$ is essentially the $(k+1)$st image
of another binary word $\infw{y}$ which then suffices for the claim by \cref{thm:k-complexities-k-image}.
\end{proof}

\begin{proof}[Proof of \cref{thm:k-complexities-k-image}]
Observe first that $\infw{x}$ is aperiodic; we shall implicitly
use this fact throughout the proof. Indeed, if it was not
aperiodic, it would be purely periodic by the recurrence
assumption. However, purely periodic words have $\bc{}{1}(n)=1$
for infinitely many $n$. This would contradict the assumption
that $\infw{x}$ has property $\tmprop(1)$.

We shall prove the claim by induction. So let first $k=1$. Then
\cref{thm:richommeTM} asserts that
there exist $a \in \{\varepsilon,0,1\}$ and a binary word
$\infw{y}$ such that $\infw{x} = a \varphi(\infw{y})$, which was to be proven.
Assume then that the claim holds for some $k$ and assume further
that $\infw{x}$ has property $\tmprop(k+1)$. It follows that
$\infw{x} = u'\varphi^k(\infw{z})$, where $u'$ is a proper (possibly empty) suffix of
$\varphi^k(0)$ or $\varphi^k(1)$.
\begin{claim}
If $\infw{z}$ has property $\tmprop(1)$, then $\infw{x}$
is of the form $\infw{x} = u\varphi^{k+1}(\infw{y})$ with
$u$ a suffix of $\varphi^{k+1}(0)$ or $\varphi^{k+1}(1)$.
\end{claim}
\begin{claimproof}
The assumption implies that $\infw{z} = b\varphi(\infw{y})$ for
some $b \in \{\varepsilon, 0, 1\}$, whence
$\infw{x} = u'\varphi^k(b)\varphi^{k+1}(\infw{y})$.
Let $y$ be a prefix of $\infw{y}$ that contains both letters $0$ and $1$. Then the factor $Y = \varphi^{k+1}(y)$ has a unique
$\varphi^{k+1}$-factorization by \cref{lem:unique_image}.
Now $u'\varphi^k(b)Y$ appears also in
$\varphi^{k+1}(\infw{y})$ due to $\infw{x}$ being recurrent. In
particular, it admits a $\varphi^{k+1}$-factorization, and since
$Y$ has a unique $\varphi^{k+1}$-factorization, we conclude that
$u'\varphi^k(b)$ must be the suffix of $\varphi^{k+1}(0)$ or
$\varphi^{k+1}(1)$.
\end{claimproof}

To prove the theorem, it is thus enough to show that
$\infw{z}$ has $\tmprop(1)$. Indeed, then $\infw{x}$ is a suffix
of the word of the form $\varphi^{k+1}(\infw{y})$ with $\infw{y}$
aperiodic, and \cref{thm:k-image-j-complexities} gives the claim.
Notice that
$\bc{\infw{x}}{k+1} = \bc{\varphi^{k}(\infw{z})}{k+1}$ again due
to recurrence of $\infw{x}$. This fact is again used
throughout the rest of the proof.

Notice now that $\infw{z}$ is also
recurrent: if it is not recurrent, then it has a prefix $w$,
containing both letters, which appears only once in $\infw{z}$. Let us write $\infw{z} = w\infw{z}'$. However, $\varphi^k(w)$ appears in $\varphi^k(\infw{z}')$ by the recurrence of $\infw{x}$. Since
$w$ contains both letters, the $\varphi^k$-factorization of
$\varphi^k(w)$ is unique. But, since $\varphi^k$ is injective,
we must find $w$ in $\infw{z}'$, a contradiction.

Let us assume towards a contradiction that $\infw{z}$ does not
have $\tmprop(1)$, and let $n$ be the least integer for which
\[
\bc{\infw{z}}{1}(n) \neq \bc{\infw{t}}{1}(n)
= 
\begin{cases}
  2, & \text{if } n \text{ is odd};\\
3, & \text{if } n \text{ is even}.
\end{cases}
\]
We now divide the proof into two cases,  depending on the parity of $n$.
As it appears, the case where $n$ is even is easier to handle.

\subsubsection{\texorpdfstring{$n$}{n} is even}
By definition of $n$, $\bc{\infw{z}}{1}(n-1)=\bc{\infw{t}}{1}(n-1)=2$ and $\bc{\infw{z}}{1}(n)\neq 3$.  Note however that $\bc{\infw{z}}{1}(n) \neq 1$ because $\infw{z}$ is aperiodic. 
Since $\bc{\infw{z}}{1}$ can increase or decrease by at most $1$ between
consecutive values, we conclude that $\bc{\infw{z}}{1}(n)=2$.

\begin{claim}\label{cl:n-greater-m-even}
We have $n > m$, where $m$ is as in \cref{prop:formula}.
\end{claim}
\begin{claimproof}
If $n > 2$ (i.e., $n \geq 4$), then $m=2$ (and $m' = 2$)  because
$\bc{\infw{z}}{1}(2) = 3$ implies that $00$,
$11 \in \Fac(\infw{z})$ while $\bc{\infw{z}}{1}(3) = 2$ implies
that $000$, $111 \notin \Fac(\infw{z})$.
If $n=2$ then $m=1$.
\end{claimproof}

We next show that $\#X_{\infw{z}}(n) \leq 5$.
To this end, let
$M_n = \max\{|u|_1 \colon u \in \Fac_{n}(\infw{z})\}$, i.e., the maximum weight among length-$n$ factors of $\infw{y}$. Assume first that $M_{n-1} = M_n$. Then any factor $v \in \Fac_{n-1}(\infw{z})$
with $|v|_1 = M_{n-1}$ is followed and preceded by $0$ (except possibly for the prefix, which is still followed by $0$), as otherwise $M_n \neq M_{n-1}$. We
conclude that
\[
X_{\infw{z}}(n) \subseteq \big\{(0,\Psi(v),0)\big\}
		\cup \left(\{0,1\} \times \big\{\Psi(v)+(1,-1)\big\} \times \{0,1\}\right),
\]
so the claim follows.

Assume second that $M_{n-1} = M_n-1$. Then each factor
$v \in \Fac_{n-1}(\infw{z})$ with $|v|_1 = M_{n-1}-1$ is
followed and preceded by $1$; this is because
$\bc{\infw{z}}{1}(n-1) = 2 = \bc{\infw{z}}{1}(n)$. In this case 
\[
X_{\infw{z}}(n) \subseteq \big\{ (1,\Psi(v),1) \big\}
		\cup \left(
			\{0,1\} \times \big\{\Psi(v) + (-1,1) \big\} \times \{0,1\}
			\right).
\]
We have shown $\#X_{\infw{z}}(n) \leq 5$.
This however leads to a contradiction: applying
formula \cref{eq:k+1-formula}, we find
\[
3\cdot 2^{k+1}-3
= \bc{\infw{t}}{k+1}(2^k n)
= \bc{\infw{x}}{k+1}(2^k n)
=
\bc{\varphi^k(\infw{z})}{k+1}(2^k n)
\leq 5 \cdot (2^k-1) + 2 = 3\cdot 2^{k+1}-2^k - 3,
\]
where the leftmost equality follows from $n$ being even and \cref{eq:k-image-j-complexities}. 
We hence move to the case where $n$ is odd.

\subsubsection{\texorpdfstring{$n$}{n} is odd}
We have that $n \geq 3$ is odd (as $\infw{z}$ is
binary). Since $\bc{\infw{z}}{1}(n-1)=3$, we have
$\bc{\infw{z}}{1}(n) \in \{3,4\}$ arguing as in the case when $n$
was even.

\begin{claim}\label{cl:3vert}
We have $\bc{\infw{z}}{1}(n) = 3$.
\end{claim}
\begin{claimproof}
Assume for a contradiction that $\bc{\infw{z}}{1}(n) = 4$. As
$\infw{z}$ recurrent, the abelian Rauzy graph $G_n$ has the
following graph as a subgraph:
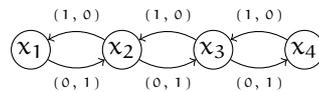
\begin{center}
\begin{tikzpicture}[scale=1.2]
\node[circle,draw,minimum size = 11pt,inner sep = 1pt] (x1) at (0,0) {$x_1$};

\node[circle,draw,minimum size = 11pt,inner sep = 1pt] (x2) at (1,0) {$x_2$};
\node[circle,draw,minimum size = 11pt,inner sep = 1pt] (x3) at (2,0) {$x_3$};
\node[circle,draw,minimum size = 11pt,inner sep = 1pt] (x4) at (3,0) {$x_4$};

\draw[->] (x1) edge[out=330,in=210] node[midway,below] {\tiny $(0,1)$} (x2);
\draw[->] (x2) edge[in=30,out=150] node[midway,above] {\tiny $(1,0)$} (x1);

\draw[->] (x2) edge[out=330,in=210] node[midway,below] {\tiny $(0,1)$} (x3);
\draw[->] (x3) edge[in=30,out=150] node[midway,above] {\tiny $(1,0)$} (x2);

\draw[->] (x3) edge[out=330,in=210] node[midway,below] {\tiny $(0,1)$} (x4);
\draw[->] (x4) edge[in=30,out=150] node[midway,above] {\tiny $(1,0)$} (x3);
\end{tikzpicture}
\end{center}
Further, since $\infw{z}$ is also aperiodic, it must have at least one loop
(at a right special vertex). We conclude that $G_n$ has at least seven edges,
that is, $\#X_{\infw{z}}(n) \geq 7$. Since $n$ is odd we have, using
\cref{eq:k+1-formula} and recalling that
$\bc{\infw{x}}{k+1} = \bc{\varphi^k(\infw{z})}{k+1}$
\[
3\cdot 2^{k+1} - 4 = \bc{\infw{t}}{k+1}(2^k n)
= \bc{\infw{x}}{k+1}(2^k n)
\geq \#X_{\infw{z}}(n)(2^k-1) + \bc{\infw{z}}{1}(n) - 2^k
\geq 6 \cdot 2^k - 3 = 3\cdot 2^{k+1} - 3,
\]
which is absurd.
\end{claimproof}

Recall the entities $m$ and $m'$ from \cref{prop:formula}.
\begin{claim}\label{cl:m2nodd}
We have $m = 2$. If $n\geq 5$, then $m' = 2$ also. Otherwise $n=3$ and $m' > 2$.
\end{claim}
\begin{claimproof}
For $n \geq 5$ one may proceed as in the proof of \cref{cl:n-greater-m-even}.
When $n=3$, we still have that $\bc{\infw{z}}{1}(2) = 3$ implies that both $00$
and $11$ appear in the word. However, $\bc{\infw{z}}{1}(3) = 3$ implies that one of
$000$ or $111$ appears in $\infw{z}$ while the other does not.
\end{claimproof}

\begin{claim}\label{cl:graph-structure}
We have that $k=1$ and $\# X_{\infw{z}}(n) = 5$.
\end{claim}
\begin{claimproof}
Consider $\bc{\infw{x}}{k+1}(2^k n)$; applying \cref{eq:k+1-formula} (using the previous claim) we have
\[
(2^k-1)\#X_{\infw{z}}(n) + \bc{\infw{z}}{1}(n)
= \bc{\infw{x}}{k+1}(2^k n)
= \bc{\infw{t}}{k+1}(2^k n)
= 3 \cdot 2^{k+1} - 4
\]
because $n$ is odd. By \cref{cl:3vert}, this is equivalent to
\[
(6 - \#X_{\infw{z}}(n))2^k + \#X_{\infw{z}}(n) = 7.
\]
Since $\infw{z}$ is aperiodic and recurrent and the abelian Rauzy graph $G_n$ has three vertices, $G_n$ must have at least five edges, i.e.,
$\#X_{\infw{z}}(n) \geq 5$. The only way to satisfy the above equality is when
$k = 1$ and $\#X_{\infw{z}}(n) = 5$.
Indeed, the function $x\mapsto (6-x)2^k + x$ is strictly decreasing (as $k \geq 1$), and for $x = 6$, it yields $6$ which is less than the
right-hand side in the above equation. Therefore we must have $\#X_{\infw{z}}(n) \leq 5$. We conclude that $\#X_{\infw{z}}(n) = 5$. Plugging this into the above equation, we find that $k=1$, as claimed.
\end{claimproof}

The previous claim shows that $G_n$ is a graph with three vertices and five
edges. Since $\infw{z}$ is recurrent and aperiodic, $G_n$ can be obtained,
without loss of generality, by adding one loop to the graph
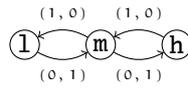
\begin{center}
\begin{tikzpicture}
\node[circle,draw,minimum size = 11pt,inner sep = 0pt] (x2) at (1,0) {\ver{l}};
\node[circle,draw,minimum size = 11pt,inner sep = 0pt] (x3) at (2,0) {\ver{m}};
\node[circle,draw,minimum size = 11pt,inner sep = 0pt] (x4) at (3,0) {\ver{h}};
\draw[->] (x2) edge[out=330,in=210] node[midway,below] {\tiny $(0,1)$} (x3);
\draw[->] (x3) edge[in=30,out=150] node[midway,above] {\tiny $(1,0)$} (x2);
\draw[->] (x3) edge[out=330,in=210] node[midway,below] {\tiny $(0,1)$} (x4);
\draw[->] (x4) edge[in=30,out=150] node[midway,above] {\tiny $(1,0)$} (x3);
\end{tikzpicture}
\end{center}
Here the leftmost vertex $\ver{l}$ corresponds to the
lightest abelian equivalence class, the rightmost vertex
$\ver{h}$ to the heaviest, and the center vertex
$\ver{m}$ to the remaining class; the
structure implies that a lightest factor precedes a heaviest
in $\infw{z}$.  
For any letter $\ver{a}\in \{\ver{l},\ver{m},\ver{h}\}$,  we shall refer to
factors of length $n$ having their Parikh-vector corresponding to $\ver{a}$ as \emph{$\ver{a}$-factors}.

Recall that $M_n$ is defined as the maximum weight among factors of length $n$.
\begin{lemma}\label{lem:graph-loop}
The graph $G_n$ contains either the loop
$\ver{h} \xrightarrow{(0,0)} \ver{h}$ or the loop $\ver{l} \xrightarrow{(1,1)} \ver{l}$.
\end{lemma}

\begin{proof}
Assume first that $M_n=M_{n-1}$. This implies that all the
heaviest factors of length $n-1$ are surrounded by $0$s in
$\infw{z}$ (meaning, preceded and followed by $0$; notice that by assumption that the prefix is not a heavy factor). Hence, there is a factor $0u0$ in
$\infw{z}$ which corresponds to the loop $\Psi(0u) \xrightarrow{(0,0)}\Psi(u0)$, where $0u$ is an $\ver{h}$-factor, because $|u0|_1 = |u|_1 = M_{n-1} = M_{n}$.

Assume then that $M_n = M_{n-1} + 1$. Since
$\bc{\infw{z}}{1}(n) = \bc{\infw{z}}{1}(n-1) = 3$, we must have
that the minimum weight of factors of length $n$ is one greater
than that of factors of length $n-1$; thus all length-$(n-1)$
minimum-weight factors are surrounded by $1$s in $\infw{z}$.
(The only exception is the prefix, which is still followed by
$1$.) So any non-prefix occurrence of such a factor (recall
$\infw{z}$ is recurrent) gives a loop on $\ver{l}$ with label
$(1,1)$ similar to the above.
\end{proof}

Applying \cref{eq:k+1-formula2} with $r=1$ and $n+1 \geq 4 > 2 = m$ (by \cref{cl:m2nodd}), we have
\[
\# Y_{\infw{z}}(n)
=\bc{\varphi(\infw{z})}{2}(2n+1)
= \bc{\infw{x}}{2}(2 n+1)
	= \bc{\infw{t}}{2}(2n+1) = 8.
\]

We show that this is impossible. Recall that we add either
of the loops $\ver{h}\xrightarrow{(0,0)} \ver{h}$
or $\ver{l}\xrightarrow{(1,1)} \ver{l}$ to the graph $G_n$.
In either case we note the following: all $\ver{l}$-factors
are followed by $1$ and all $\ver{h}$-factors are followed by $0$.
In particular, inspecting factors of length $n+1$, we only have two distinct
Parikh-vectors. Therefore, the graph $G_{n+1}$ has only two vertices, i.e., $\bc{\infw{z}}{1}(n+1) = 2$.
The number of edges of such a graph is at most six: both vertices can have two loops and one outgoing edge to the other vertex. However,
applying \cref{eq:k+1-formula} (with $n+1 > m$), we find
$\#X_{\infw{z}}(n+1) + 2 = 9$ because $n+1$ is even. But then
$\#X_{\infw{z}}(n+1) = 7$, which is impossible. 
This final contradiction proves that $n$ cannot be odd either.
This concludes the proof of \cref{thm:k-complexities-k-image}.
\end{proof}

\section{Several Answers to \texorpdfstring{\cref{q: strict ineq}}{}}
\label{sec:question_of_Lejeune}

One can give a rather direct answer to \cref{q: strict ineq}.
Indeed, let $\infw{c}$ be the binary Champernowne word, that is,
the concatenation of the binary representations of the non-negative integers: $0$, $1$, $10$, $11$, $100$, $101$, $110$, $111$, \ldots.
Notice that $\infw{c}$ contains all binary words. For each $k$,
there exist two binary words $u,v$ such that $u\sim_{k}v$ and $u\not\sim_{k+1}v$ (see, for instance, \cref{thm:Ochsenschlager}). Therefore, the same properties hold for $ux$ and $vx$, for all $x\in\{0,1\}^*$, thus 
$\bc{\infw{c}}{k} \prec \bc{\infw{c}}{k+1}$ for all $k$. Clearly $\bc{\infw{c}}{1}(n)=n+1$ is unbounded and so is $\bc{\infw{c}}{k}$ for $k\geq 2$.

Observe that $\infw{c}$ is not morphic, nor \emph{uniformly recurrent} (a word $\infw{x}$ is uniformly
recurrent if for each $x \in \mathcal{L}(\infw{x})$ there exists
$N \in \N$ such that $x$ appears in all factors in
$\Fac_N(\infw{x})$). Therefore in the rest of the section we provide more ``structured'' words answering \cref{q: strict ineq}.

\subsection{A Non-Binary Pure Morphic Answer}

Consider the morphism $g\colon \{a,0,1,\alpha\}^* \to \{a,0,1,\alpha\}^*$ defined by 
\[ 
				a\mapsto a0\alpha,\ 
				0\mapsto \varphi(0),\ 
				1\mapsto \varphi(1),\ 
				\alpha \mapsto \alpha^2
\]
where $\varphi$ is the Thue--Morse morphism.
We have $\infw{g} = g^{\omega}(a) = a\prod_{j=0}^{\infty}\varphi^j(0) \alpha^{2^j}$.
We show that the word $\infw{g}$ answers \cref{q: strict ineq}:

\begin{proposition}
The abelian complexity of $\infw{g}$ is unbounded and $\bc{\infw{g}}{k} \prec \bc{\infw{g}}{k+1}$ for all $k \geq 1$.
\end{proposition}

\begin{proof}
The abelian complexity of $\infw{g}$ is (at least) linear, since
\[\{|u|_{\alpha} \colon u \in \Fac_n(\infw{g})\} = \{0,\ldots,n\}.\]
Furthermore,
for each $k \in \N$ there exist infinitely many words $u_n$, $v_n \in \mathcal{L}(\infw{g})$ such that
$u_n\sim_k v_n$ but $u_n\not\sim_{k+1} v_n$: by
\cref{thm:Ochsenschlager}, take
$u_n=\varphi^k(0)\alpha^n$ and $v_n=\varphi^k(1)\alpha^n$. Consequently
$\bc{\infw{g}}{k} \prec \bc{\infw{g}}{k+1}$ for all
$k\geq 1$.
\end{proof}

\subsection{A Binary Morphic Answer}

Consider the word $\tau(\infw{g})$, where $\infw{g}$ is the word
defined in the previous subsection, and $\tau$ is the coding
$a \mapsto \varepsilon$, $0 \mapsto 0$, $1\mapsto 1$, and
$\alpha \mapsto 1$.
We have the following:

\begin{proposition}
The abelian complexity of $\tau(\infw{g})$ is unbounded and $\bc{\tau(\infw{g})}{k} \prec \bc{\tau(\infw{g})}{k+1}$ for all $k \geq 1$.
\end{proposition}
\begin{proof}
The word $\tau(\infw{g})$ has unbounded abelian complexity: it
contains arbitrarily long words $u$ for which $|u|_1 = \left\lfloor |u|/2 \right\rfloor$ (take factors of the Thue--Morse word for instance). Similarly it contains arbitrarily long powers of $1$. Consequently, the word has unbounded abelian complexity (recall \cref{rk:abelian_complexity}).

To show $\bc{\tau(\infw{g})}{k} \prec \bc{\tau(\infw{g})}{k+1}$ for all $k$, we notice that the
same arguments as in the case of $\infw{g}$
can be applied verbatim with $\tau(u_n)$ and $\tau(v_n)$.
\end{proof}

\subsection{A Binary Uniformly Recurrent Answer}
We note that none of the above words are
uniformly recurrent. A natural candidate for such a word is one
that has relatively high factor complexity. Uniformly recurrent
words having positive \emph{topological entropy}%
\footnote{The topological entropy of a word $\infw{x}$ is defined
as the quantity
$\smash{\displaystyle{\lim_{n\to \infty}}}\tfrac{\log\fc{\infw{x}}(n)}{n}$, which
exists for any $\infw{x}$ (see \cite[\S 4.3.2]{CANT2010}).}
were studied by Grillenberger in~\cite{Grillenberger1973}.
A construction for uniformly recurrent positive entropy words
appears in \cite[\S 4.4.3]{CANT2010}; this construction is
simpler than that of Grillenberger's, though some properties are
lost (see \cite[\S 4.4.3]{CANT2010} for a discussion). We recall
this construction here. To attain a word with entropy between
$0$ and $\log d$, define $D_{0} = \{0,1,\ldots,d-1\}$ and
let $(q_k)_{k\geq 0}$ be a sequence of positive integers.
Assuming $D_k$ is constructed, let $u_k$ be the product of words
of $D_k$ in lexicographic order (assuming, e.g.,
$0<1<\ldots<d-1$).
Define then $D_{k+1}:= u_k D_k^{q_k}$. The sequence
$(u_k)_{k\in\N}$ converges to a uniformly recurrent word
$\infw{u}$ having, with a suitable choice of $(q_k)$, the
prescribed entropy. We consider the word with $d = 2$ and
$q_k = 2$ for all $k$ (and are not interested in the entropy). Hence for us $\infw{u}=0100010101100111\cdots$.

\begin{lemma}
Let $k\geq 1$.
If, for some $j \geq 0$, $D_j$ contains two words $u$, $v$, such
that $u \sim_k v$ and $u \not\sim_{k+1} v$, then $D_{j+1}$ contains words $x$, $y$, $z$ and $w$ such that
\begin{itemize}[label=$\bullet$]
\item $x \sim_{k} y$ but $x\not \sim_{k+1}y$;
\item $z \sim_{k+1} w$ but $z \not \sim_{k+2} w$.
\end{itemize}
\end{lemma}
\begin{proof}
By definition, the set $D_{j+1}$ contains the words $x = u_juu$, $y = u_jvv$, $z = u_juv$, and $w = u_jvu$.

We first consider the pair $x,y$.
Since $\sim_k$ is a congruence, $x \sim_k y$.
To see that $x \not\sim_{k+1} y$, assume the contrary, so that
this equivalence reduces to $uu \sim_{k+1} vv$ by \cref{lem: cancellation property}.
\cref{lem:diff-powers} implies $u \sim_{k+1} v$, a contradiction.

Next we have $uv \sim_{k+1} vu$ by \cref{thm:w35},
and thus $z = u_juv \sim_{k+1} u_jvu = w$ by \cref{lem: cancellation property}. Similarly $z \sim_{k+2} w$
would imply $uv \sim_{k+2} vu$ and thus $u \sim_{k+1} v$ by \cref{thm:w35}, a contradiction. The claim follows. 
\end{proof}

\begin{theorem}
The abelian complexity of $\infw{u}$ is unbounded and $\bc{\infw{u}}{k} \prec \bc{\infw{u}}{k+1}$ for all $k \geq 1$.
\end{theorem}
\begin{proof}
First we show that
$\bc{\infw{u}}{1}$ is unbounded.
Assume, for some $j\geq 0$, that $D_j$ contains words $u,v$ with $|u|_0 - |v|_0 = 2^{j}$ (this holds for $j=0$). Then by definition
$D_{j+1}$ contains the words $x = u_juu$ and $y = u_jvv$, for which
$|x|_0 - |y|_0 = 2 (|u|_0 - |v|_0) = 2^{j+1}$. This observation suffices for the claim by \cref{rk:abelian_complexity}.

We then prove the second part of the statement.
Observe that $D_1$ contains the words $0101$ and $0110$, which are abelian
equivalent, but not $2$-binomially equivalent (as $\binom{0101}{01}=3$ and $\binom{0110}{01}=2$). The above lemma then
implies that for all $k \geq 1$ and for all $j\geq k$, the set $D_j$
contains words that are $k$-binomially equivalent, but not $(k+1)$-binomially equivalent.
The claim follows.
\end{proof}

\begin{remark}
It can be shown that the word $\infw{u}$ above has topological
entropy equal to $0$. By modifying the arguments above suitably,
the statement of the above theorem holds for any choice of $d$
and $(q_k)$ in the construction---as long as $q_k > 1$ for
infinitely many $k$. Note that to attain a word with positive
entropy, the sequence $(q_k)$ must satisfy this property. Hence
we have:
\emph{For any positive real number $h$ there
is a uniformly recurrent $d$-ary word (with
$d = \max\{2,\lfloor h \rfloor + 1\}$) having entropy~$h$,
unbounded $\bc{}{1}$, and $\bc{}{k}\prec \bc{}{k+1}$ for all $k$.}
\end{remark}

\section{Answer to \texorpdfstring{\cref{q: stab}}{} and Beyond}\label{q: stab-bis}

The word $0^{\omega}$ gives
$\bc{}{1} = \fc{}$. 
The Fibonacci
word $\infw{f} = 0100101001001010010\cdots$, the fixed point of
the morphism $0 \mapsto 01, 1\mapsto 0$, is a pure morphic word such that $2=\bc{\infw{f}}{1} \prec \bc{\infw{f}}{2} = \fc{\infw{f}}$ by \cref{thm:Sturmian2-binomial}.

\begin{remark}\label{rem:abelian_neq_factor}
We notice that $\bc{\infw{x}}{1} = \fc{\infw{x}}$ cannot be attained
for an aperiodic word $\infw{x}$ (indeed, there must exist a factor
$ava$, with $a \in A$ and $v$ containing a letter different to $a$,
whence $av \sim_1 va$ with $av \neq va$). In fact, the only ultimately periodic words over an $m$-letter alphabet $\{a_1,\ldots,a_m\}$
for which the equality holds are of the form $a_1^{n_1} a_2^{n_2}\cdots a_m^{\omega}$, $n_i \in \N$ (up to permutation of the letters).
\end{remark}

 To answer \cref{q: stab} for larger values of $k$, we take images of a Sturmian word $\infw{s}$ by a power of the Thue--Morse morphism $\varphi$ and we prove the following result.

\begin{theorem}\label{thm:iterate_thue-morse_fibonacci} Let $\varphi$ be the Thue--Morse morphism. Let $\infw{s}$ be a Sturmian word. 
For each $k \geq 0$, the word $\infw{s}_k := \varphi^{k}(\infw{s})$ has
\[\bc{\infw{s}_k}{1} \prec \bc{\infw{s}_k}{2} \prec \cdots \prec \bc{\infw{s}_k}{k+1} \prec \bc{\infw{s}_k}{k+2} = \fc{\infw{s}_k}.\]
In particular, putting the Fibonacci word for $\infw{s}$ gives a morphic positive answer to \cref{q: stab}.
\end{theorem}
\begin{proof}
Observe that $\infw{s}_k$ has bounded $(k+1)$-binomial complexity as a straightforward application of \cref{thm:Parikh-collinear_bounded-k_to_bounded_k+1} (because $\infw{s}$ has bounded abelian complexity), and thus $\bc{\infw{s}_k}{k+1} \prec \fc{\infw{s}_k}$.
By \cref{cor:k-image_k+1-prec}, we need only
to show that $\bc{\infw{s}_k}{k+2} = \fc{\infw{s}_k}$. 

Let $u$ and $v$ be distinct factors of $\infw{s}_k$. Assume they
are $(k+2)$-binomially equivalent.
By \cref{pro:k-image_k+1-prefix-suffix}, we have
that $u = p\varphi^k(z)s$, $v = p\varphi^k(z')s$ with $z \sim_1 z'$.
If $z \neq z'$, then $z \not \sim_2 z'$ by
\cref{thm:Sturmian2-binomial}. But then \cref{lem:Michel's} implies that
$\varphi^k(z) \not\sim_{k+2} \varphi^{k}(z')$, contradicting the
assumption. Hence we deduce that $z = z'$, but then $u = v$ contrary to the assumption.
\end{proof}

\begin{remark}
In the above proof, since $\infw{s}$ is Sturmian, \cref{thm:Sturmian2-binomial} says distinct factors are not $2$-binomially equivalent. This means that \cref{thm:iterate_thue-morse_fibonacci} applies to and only to aperiodic words $\infw{s}$ such that $\bc{\infw{s}}{2}=\fc{\infw{s}}$.
The ``only if''-part of the statement follows by a repeated application
of \cref{prop:Parikh-collinear-k-characterization} on the Thue--Morse morphism together with the fact that the morphism is injective.
\end{remark}


\subsection{Strengthening \texorpdfstring{\cref{q: stab}}{Question B}}\label{sec: 7.1}

We answered \cref{q: stab} by providing a word with bounded abelian complexity. We can therefore strengthen the question with the following extra requirement.

\begin{questions}\label{q: stab + unbounded}
For each $\ell \ge 1$, does there exist a word $\infw{w}$ (depending on $\ell$) such that $\bc{\infw{w}}{1}$ is unbounded and 
\[
\bc{\infw{w}}{1}
\prec \bc{\infw{w}}{2}
\prec \cdots 
\prec \bc{\infw{w}}{\ell -1}
\prec \bc{\infw{w}}{\ell}
= \fc{\infw{w}}?
\]
If the answer is positive, can we find a (pure) morphic such word $\infw{w}$?
\end{questions}

The following word answers the question for $\ell=3$ in the positive.

\begin{theorem}
The word $\infw{h} = 0112122122212222122222\cdots$ fixed point of the morphism
$0\mapsto 01$, $1 \mapsto 12$, and $2\mapsto 2$ is such that its abelian complexity $\bc{\infw{h}}{1}$ is unbounded and $\bc{\infw{h}}{1} \prec \bc{\infw{h}}{2} \prec \bc{\infw{h}}{3} = \fc{\infw{h}}$.
\end{theorem}

We obtain the previous theorem by combining the following two results.

\begin{proposition}\label{prop:ab-2-bin_h}
The abelian complexity $\bc{\infw{h}}{1}$ of $\infw{h}$ is unbounded and
$\bc{\infw{h}}{1}(n) < \bc{\infw{h}}{2}(n) < \fc{\infw{h}}(n)$ for all $n\geq 6$.
\end{proposition}
\begin{proof}
We claim that $\bc{\infw{h}}{1}$ is of the order $\Theta(\sqrt{n})$. Clearly it suffices to show the claim for the
word $\infw{h}'=0^{-1}\infw{h}$, as removing the first zero always removes exactly one abelian equivalence class: the only one that contains a zero. The
resulting word $\infw{h}'$ is effectively a binary word;
it is evident that the maximal number of $1$'s in a word of length $n$ is attained by the prefix of
$\infw{h}'$. This value equals the maximal $m$ for which $\sum_{i=1}^m i = \binom{m+1}{2} \leq n$. Clearly $m = \Theta(\sqrt{n})$.
By \cref{rk:abelian_complexity}, we conclude that the abelian complexity of $\infw{h}$ is $\Theta(\sqrt{n})$.

Since the abelian complexity of $\infw{h}$ if unbounded, so is its $2$-binomial complexity. However, the $2$-binomial complexity does not equal
the factor complexity at lengths $n \geq 6$: $\infw{h}$ contains both the factors $12^{n-2}1$ and $212^{n-4}12$ which are readily seen to be $2$-binomially equivalent. (One may also invoke a result from \cite{Fosse2004} for binary alphabets.)

Finally observe that the abelian complexity does not coincide with the $2$-binomial complexity either:
the factors $2^x 1 2^y$ with $x+y = n-1$ are abelian equivalent but not $2$-binomially equivalent.
This ends the proof.
\end{proof}

\begin{proposition}\label{prop:3-bin_factor_h}
We have $\bc{\infw{h}}{3} = \fc{\infw{h}}$.
\end{proposition}
\begin{proof}
We may again discard the first $0$ of $\infw{h}$,
as the prefix is the only factor containing a zero. Assume to the contrary that there exist $3$-binomially equivalent distinct factors $u_1$ and $u_2$ in
$\infw{h}'=0^{-1}\infw{h}$. The two factors must contain the same number of $1$'s, and hence at least one under the assumption that they are distinct. If the factors are of the form $u_i = 2^{x_i}12^{y_i}$ with $x_1 \neq x_2$, then the factors are not even $2$-binomially equivalent. So the words
contain at least two $1$'s. By the structure of $\infw{h}$, we may write $u_i = 2^{x_i}12^{a_i}12^{a_i+1}1\cdots 12^{a_i+t}12^{y_i}$ for some
$t\geq 0$, $a_i \in \mathbb{N}$, $x_i < a_i$ and $y_i \leq a_{i}+t+1$ for all $i\in\{1,2\}$.
If $a_1 = a_2$, then $x_1 \neq x_2$, and we again deduce that the factors are not even $2$-binomially equivalent. So we must have $a_1 < a_2$ without loss of generality.
We show that in this case the factors are not $3$-binomially equivalent.
Indeed, consider the coefficient $\binom{\cdot}{121}$. For $i=1,2$, we clearly have
\begin{equation}\label{eq:rel-coeff121}
\binom{u_i}{121} = \binom{v_i}{121},
\end{equation}
where $v_i = 12^{a_i}12^{a_i+1}1\cdots 12^{a_i+t}1$ is obtained from $u_i$ by deleting a prefix and a suffix.
But, since $a_1 < a_2$, notice now that $v_1$ is a proper subword of $v_2$, meaning that each occurrence of $121$ in $v_1$ has a corresponding occurrence
in $v_2$. Clearly $v_2$ will have more occurrences of $121$.
This combined with~\eqref{eq:rel-coeff121} gives the claim.
\end{proof}

A complete answer to \cref{q: stab + unbounded} is far from obvious; especially if one wishes to obtain a pure morphic word. Conversely, for a non-periodic morphic word $\infw{w}$ which is not the fixed point of a Parikh-collinear morphism, one can wonder about the existence of a minimal value $m$ for which the binomial and factor complexities would coincide. Does there exists $m\in \N$ such that $\bc{\infw{w}}{m}=\fc{\infw{w}}$?

Even with an apparently simple situation, it is far from obvious. As stated in the introduction, computing the $k$-binomial complexity of a particular infinite word remains challenging. The period-doubling word $\infw{pd}=0100 0101 010001\cdots$, the fixed point of $\sigma:0\mapsto 01$, $1\mapsto 00$, can be proved to have the following properties. Its abelian complexity $\bc{\infw{pd}}{1}$ is unbounded \cite[Lem. 4]{KarhumakiSaarela2017}.  For the $2$-binomial complexity, we have
$\bc{\infw{pd}}{2}(2^n) = \fc{\infw{pd}}(2^n)$ for all $n$, but
$\bc{\infw{pd}}{2}(n) < \fc{\infw{pd}}(n)$ for all $n \neq 2^m$ \cite[Prop. 4.5.1]{Lejeune2021thesis}. Otherwise stated, $\bc{\infw{pd}}{1} \prec \bc{\infw{pd}}{2} \prec \fc{\infw{pd}}$.  Computer experiments show that $\bc{\infw{pd}}{3} \prec \fc{\infw{pd}}$  and suggest that $\bc{\infw{pd}}{4}= \fc{\infw{pd}}$.

\subsection{Completing the Binomial Complexities of \texorpdfstring{$\varphi^k$}{} Applied to a Sturmian word}
For any $k\geq 1$, the results presented so far imply that we
have the exact $j$-binomial complexity function of $\varphi^k(\infw{s})$, with $\infw{s}$
a Sturmian word, for each $j \neq k+1$. As a bonus, we compute the
$(k+1)$-binomial complexity in \cref{prop:sturmiancomplexity}.
We first analyze the abelian Rauzy graphs of Sturmian words,
after which we may apply \cref{prop:formula}
to obtain the exact $(k+1)$-binomial complexity as well.

A Sturmian word $\infw{s}$ has $\bc{\infw{s}}{1}(n) = 2$ for all $n \geq 1$.
Hence its abelian Rauzy graph has two vertices.

\begin{proposition}\label{prop:sturmian-abelian-Rauzy-graphs}
Let $G_n = (V_n,E_n)$. We have $\#E_1 = 3$ and $\#E_n = 4$ for all $n\geq 2$. For all $n\geq 1$, we have $\#E_n/{\equiv_L} + \#E_n/{\equiv_R} = 6$.
\end{proposition}
\begin{proof}
Sturmian words are aperiodic, so the graph $G_n$ is always strongly connected.
It also always has a right special vertex.

The claim is plain to verify for $G_1$. Indeed, only one of the vertices can
have a loop, and this loop can only be labeled with $(a,a)$, where $a$ is the
letter corresponding to the vertex in question. The second claim is
straightforward to check in this case.

We next consider $G_n$, $n\geq 2$. Since $G_n$ is strongly connected and
has a right special vertex, we conclude that $G_n$ is obtained by adding
(possibly zero) loops to one of the graphs in \cref{fig:sturmian-subgraph}

\begin{figure}[h!t]
\centering
\begin{tikzpicture}
\node[circle,draw,minimum size = 11pt,inner sep = 0pt] (x) at (1,0) {$\vec{x}$};
\node[circle,draw,minimum size = 11pt,inner sep = 0pt] (y) at (2,0) {$\vec{y}$};

\draw[->] (x) edge[out=330,in=210] node[midway,below] {\tiny $(0,1)$} (y);
\draw[->] (y) edge[in=30,out=150] node[midway,above] {\tiny $(1,0)$} (x);

\draw[->] (x) edge[out=70,in=110,looseness=8] node[midway,above] {\tiny $(0,0)$} (x);
\end{tikzpicture}
\raisebox{15pt}{
\quad
or
\quad
}
\begin{tikzpicture}
\node[circle,draw,minimum size = 11pt,inner sep = 0pt] (x) at (1,0) {$\vec{x}$};
\node[circle,draw,minimum size = 11pt,inner sep = 0pt] (y) at (2,0) {$\vec{y}$};

\draw[->] (x) edge[out=330,in=210] node[midway,below] {\tiny $(0,1)$} (y);
\draw[->] (y) edge[in=30,out=150] node[midway,above] {\tiny $(1,0)$} (x);

\draw[->] (y) edge[out=70,in=110,looseness=8] node[midway,above] {\tiny $(1,1)$} (y);
\end{tikzpicture}
\caption{Possible subgraphs of a Sturmian word's abelian Rauzy graph.}
\label{fig:sturmian-subgraph}
\end{figure}
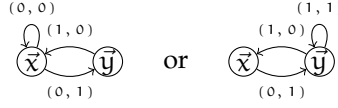

By symmetry, we may assume it is the one on the left.
We claim that we add exactly one loop to this graph. More precisely,
the added loop is either $\vec{x} \xrightarrow{(1,1)} \vec{x}$ or
$\vec{y} \xrightarrow{(0,0)} \vec{y}$.

First off, we cannot add the loop $\vec{y} \xrightarrow{(1,1)} \vec{y}$;
otherwise we have the factors $x0$ and $y1$, where $x$ is an $\vec{x}$-factor
and $y$ is a $\vec{y}$-factors, for which we have $|y1|_1 - |x0|_1 = 2$
contradicting balancedness at $n+1$.
Towards a contradiction, we consider whether we add neither or both of the remaining admissible loops.

Assume first that we add both loops. Then we have that length $n+1$ factors
$1x1$ and $0y0$, where $1x$, $x1$ are $\vec{x}$-factors and $0y$, $y0$ are
$\vec{y}$-factors, having $|1x1|_1 = |0y0|_1$. However, we then have
$|y|_1 - |x|_1 = 2$, which gives a contradiction with balancedness at $n-1$.

To complete the proof of the first claim, suppose we add neither of the loops.
Inspecting $G_n$, we see that a factor of length $n$ is an $\vec{x}$-factor if and only if it begins with $0$. Consider the right special factor $0v$ of length $n$ (i.e., $|v| = n-1 \geq 1$) (it begins with $0$ by the form of the graph).
Since $v1$ is a $\vec{y}$-factor, we deduce that $v$ begins with a $1$. But then $v0$ is a $\vec{x}$-factor beginning with $1$, a contradiction.

The second part of the claim is now straightforward. The edges in
the left-hand graph in
\cref{fig:sturmian-subgraph} are all pairwise inequivalent under both
$\equiv_L$ and $\equiv_R$. Both of the admissible loops to be added to
the graph to obtain $G_n$ are equivalent to the non-loop edges of the graph.
Hence $\#E_n/{\equiv_L} + \#E_n/{\equiv_R} = 6$.
\end{proof}


For a Sturmian word $\infw{s}$, we have, by \cref{prop:sturmian-abelian-Rauzy-graphs}, $X_{\infw{s}}(1) = 3$, $X_{\infw{s}}(n) = 4$ for all $n\geq 2$, and $Y_{\infw{s}}(n) = 6$ for all $n\geq 1$. We also have that
$1 = m < m'$ using the notation of \cref{prop:formula}. Hence,
applying the proposition, we have:
\begin{proposition}\label{prop:sturmiancomplexity}
For all $n\geq 0$ and $0\leq r < 2^k$
\begin{equation}\label{eq:k+1-sturmian}
\bc{\varphi^k(\infw{s})}{k+1}(2^kn + r)
= \begin{cases}
	\fc{\infw{t}}(r), &\text{if }n=0 \text{ and } 0\leq r < 2^k;\\
	3\cdot 2^{k} - 2, &\text{if }n=1 \text{ and } r=0;\\
		3 \cdot 2^k + r -1,  &\text{if }n=1 \text{ and } r>0;\\
2^{k+2} - 2, &\text{otherwise}.
\end{cases}
\end{equation}
\end{proposition}

We thus conclude the following. For any $k\geq 1$ and a Sturmian
word $\infw{s}$, we have $\bc{\varphi^k(\infw{s})}{j} = \bc{\infw{t}}{j}$ if $1\leq j\leq k$ (\cref{thm:k-image-j-complexities});
$\bc{\varphi^k(\infw{s})}{k+1}$ is as in \cref{eq:k+1-sturmian};
and $\bc{\varphi^k(\infw{s})}{j} = \fc{\varphi^k(\infw{s})}$ if $j \geq k+2$ (\cref{thm:iterate_thue-morse_fibonacci}). The exact
value for $\fc{\varphi^k(\infw{s})}(n)$ is given by
$\fc{\infw{t}}(n)$ when $n\leq 2^k$ (\cref{lem:k-image-short-factors}), and by $n+2^{k+1}-1$ for
$n>2^k$. The latter can be deduced by using the methods described in
\cite[\S 4.1]{Frid1999}.

 \bibliographystyle{plainurl}
\bibliography{../bibliography.bib}

\end{document}